\documentclass[reqno,a4paper]{amsart}
\usepackage{amssymb}
\usepackage{amsmath}
\begin{document} 
\newtheorem{prop}{Proposition}[section]
\newtheorem{Def}{Definition}[section] \newtheorem{theorem}{Theorem}[section]
\newtheorem{lemma}{Lemma}[section] \newtheorem{Cor}{Corollary}[section]

\title[ MKG in 2D]{\bf Low regularity solutions for the (2+1) - dimensional Maxwell-Klein-Gordon equations in temporal gauge}
\author[Hartmut Pecher]{
{\bf Hartmut Pecher}\\
Fachbereich Mathematik und Naturwissenschaften\\
Bergische Universit\"at Wuppertal\\
Gau{\ss}str.  20\\
42119 Wuppertal\\
Germany\\
e-mail {\tt pecher@math.uni-wuppertal.de}}
\date{}

\begin{abstract}
The Maxwell-Klein-Gordon equations in 2+1 dimensions in temporal gauge are locally well-posed for low regularity data even below energy level. The corresponding (3+1)-dimensional case was considered by Yuan. Fundamental for the proof is a partial null structure in the nonlinearity which allows to rely on bilinear estimates in wave-Sobolev spaces by d'Ancona, Foschi and Selberg, on an $(L^p_x L^q_t)$ - estimate for the solution of the wave equation, and on the proof of a related result for the Yang-Mills equations by Tao.

\end{abstract}
\maketitle
\renewcommand{\thefootnote}{\fnsymbol{footnote}}
\footnotetext{\hspace{-1.5em}{\it 2010 Mathematics Subject Classification:} 
35Q40, 35L70 }
\normalsize 
\setcounter{section}{0}
\section{Introduction and main results}
\noindent Consider the Maxwell-Klein-Gordon equations
\begin{align}
\label{1.1}
\partial^{\alpha} F_{\alpha \beta} & = -Im(\phi \overline{D_{\beta} \phi}) \\
\label{1.2}
D^{\mu}D_{\mu} \phi &= m^2 \phi
\end{align}
in Minkowski space $\mathbb{R}^{1+2} = \mathbb{R}_t \times \mathbb{R}^2_x$ with metric $diag(-1,1,1)$. Greek indices run over $\{0,1,2\}$, Latin indices over $\{1,2\}$, and the usual summation convention is used.  Here $m \in \mathbb{R}$ and
$$ \phi: \mathbb{R}\times \mathbb{R}^2 \to \mathbb{C} \, , \, A_{\alpha}: \mathbb{R} \times \mathbb{R}^2 \to \mathbb{R} \, , \, F_{\alpha \beta} = \partial_{\alpha} A_{\beta} - \partial_{\beta} A_{\alpha} \, , \, D_{\mu} = \partial_{\mu} + iA_{\mu} \, . $$
$A_{\mu}$ are the gauge potentials, $F_{\mu \nu}$ is the curvature. We use the notation $\partial_{\mu} = \frac{\partial}{\partial x_{\mu}}$, where we write $(x^0,x^1,x^2)=(t,x^1,x^2)$ and also $\partial_0 = \partial_t$.

Setting $\beta =0$ in (\ref{1.1}) we obtain the Gauss-law constraint
\begin{equation}
\label{1.3}
\partial^j F_{j 0} = -Im(\phi \overline{D_0 \phi})  \, .
\end{equation}
The system (\ref{1.1}),(\ref{1.2}) is invariant under the gauge transformations
$$A_{\mu} \to A'_{\mu} = A_{\mu} + \partial_{\mu} \chi \, , \, \phi \to \phi' = e^{i \chi} \phi \, , \, D_{\mu} \to D'_{\mu} = \partial_{\mu} + i A'_{\mu} \, .$$
This allows to impose an additional gauge condition. We exclusively consider the temporal gauge
\begin{equation}
\label{1.5}
A_0 = 0 \, .
\end{equation}
In this gauge the system (\ref{1.1}),(\ref{1.2}) is equivalent to
\begin{align}
\label{1.6}
\partial_t \partial^j A_j &=  Im(\phi \overline{\partial_t \phi}) \\
\Box A_j &= \partial_j(\partial^k A_k) - Im(\phi \overline{\partial_j \phi}) +  A_j |\phi|^2 \\
\Box \phi &= -i (\partial^k A_k) \phi - 2i A^k \partial_k \phi + A^k A_k \phi  + m^2 \phi\, ,
\end{align}
where $\Box = -\partial_t^2 + \Delta$ is the d'Alembert operator. 

Other choices of the gauge are the Coulomb gauge $\partial^j A_j =0$ and the Lorenz gauge $\partial^{\mu} A_{\mu} = 0$. 

Let us make some historical remarks. Most of the results were given in 3+1 dimensions. Klainerman and Machedon \cite{KM} proved global well-posedness in energy space in Coulomb gauge and temporal gauge. Local well-posedness in Coulomb gauge for data for $\phi$ in the Sobolev space $H^s$ and for $A$ in $H^r$ with $r=s > 1/2$, i.e., almost down to the critical space with repect to scaling, was shown by Machedon and Sterbenz \cite{MS}. In Lorenz gauge the global well-posedness result in energy space is due to Selberg and Teshafun \cite{ST}. The author \cite{P} proved local well-posedness for $s=\frac{3}{4}+\epsilon$ and $r=\frac{1}{2}+\epsilon$. In temporal gauge Yuan \cite{Y} obtained local well-posedness for $s=r > \frac{3}{4}$ in $X^{s,b}$-spaces and global well-posedness for $s=r=1$. The author \cite{P1} proved that the finite energy solutions are also unique in the natural solution spaces. These results in temporal gauge rely on a similar result by Tao \cite{T1} for the Yang-Mills equations and small data. 

In 2+1 dimensions Moncrief \cite{M} proved global well-posedness in Lorenz gauge for data in $H^2$. Local well-posedness in Lorenz gauge for $s=\frac{3}{4}+\epsilon$ and $r=\frac{1}{4}+\epsilon$ was shown by the author \cite{P}. In Coulomb gauge local well-posedness for $s=r=\frac{1}{2}+\epsilon$ and also for $s=\frac{5}{8}+\epsilon$ , $r=\frac{1}{4}+\epsilon$ was obtained by Czubak and Pikula \cite{CP}.

In the present paper we exclusively consider the (2+1)-dimensional case in the temporal gauge and prove local well-posedness for data under minimal regularity assumptions. We need $\phi(0) \in H^s$ , $(\partial_t \phi)(0) \in H^{s-1}$ , $A^{df}(0) \in  H^r$ , $(\partial_t A^{df})(0) \in H^{r-1}$ , $|\nabla|^{\tilde{\epsilon}} A^{cf}(0) \in H^{l-\tilde{\epsilon}}$ , where $A^{df}$ and $A^{cf}$ denote the divergence-free and "curl-free" part of $A$, respectively, where an admissible choice is $s=l=\frac{1}{2}+\frac{1}{14}+,$  $r = \frac{1}{4}+$ , and also $s=r=l=\frac{1}{2}+\frac{1}{12}+.$  Uniqueness holds in spaces of Bourgain-Klainerman-Machedon type. If $s=r=l=1$ we even obtain unconditional uniqueness in the natural solution spaces. For a precise statement we refer to Theorem \ref{Theorem}.  We make use of a partial null structure of the nonlinearities and use bilinear estimates in wave-Sobolev spaces which were given systematically by d'Ancona, Foschi and Selberg \cite{AFS}. We also need a powerful variant of Strichartz' estimates which gives an estimate for the $L^6_x L^2_t$-norm of the solution of the wave equation which goes back to Tataru \cite{KMBT}. The 3-dimensional variant was used by Tao \cite{T1} for the more general Yang-Mills equation. Tao's hybrid estimates in this paper for the product of functions in wave-Sobolev spaces $X^{s,b}_{|\tau|=|\xi|}$ and in product Sobolev spaces $X^{s,b}_{\tau = 0}$ (cf. the definition of the spaces below) are fundamental for our calculations.

We denote both the Fourier transform with respect to space and time and with respect to space by $\,\widehat{\cdot}\,\,$. The operator
$|\nabla|^{\alpha}$ is defined by $({\mathcal F}(|\nabla|^{\alpha} f))(\xi) = |\xi|^{\alpha} ({\mathcal F}f)(\xi)$ and similarly $ \langle \nabla \rangle^{\alpha}$ , where $\langle \,\cdot\, \rangle := (1+|\cdot|^2)^{\frac{1}{2}}$ . The inhomogeneous and homogeneous Sobolev spaces are denoted by $H^{s,p}$ and $\dot{H}^{s,p}$ , respectively. For $p=2$ we simply denote them by $H^s$ and $\dot{H}^s$. We repeatedly use the Sobolev embeddings $H^{s,p} \hookrightarrow L^q$ for $\frac{1}{p} \ge \frac{1}{q} \ge \frac{1}{p}-\frac{s}{2},$ and $\dot{H}^{s,p} \hookrightarrow L^q$  for $\frac{1}{q} = \frac{1}{p}-\frac{s}{2}$ , and $1<p\le q < \infty$ . We also use the notation
$a \pm := a \pm \epsilon$ for a sufficiently small $\, \epsilon >0$ , so that $a-- < a- < a<a+<a++$ .

The standard space $X^{s,b}_{\pm}$ of Bourgain-Klainerman-Machedon type (which was already considered by M. Beals \cite{B}) belonging to the half waves is the completion of the Schwarz space  $\mathcal{S}({\mathbb R}^3)$ with respect to the norm
$$ \|u\|_{X^{s,b}_{\pm}} = \| \langle \xi \rangle^s \langle  \tau \pm |\xi| \rangle^b \widehat{u}(\tau,\xi) \|_{L^2_{\tau \xi}} \, . $$ 
Similarly we define the wave-Sobolev space $X^{s,b}_{|\tau|=|\xi|}$ with norm

$$ \|u\|_{X^{s,b}_{|\tau|=|\xi|}} = \| \langle \xi \rangle^s \langle  |\tau| - |\xi| \rangle^b \widehat{u}(\tau,\xi) \|_{L^2_{\tau \xi}}  $$
and also $X^{s,b}_{\tau =0}$ with norm 
$$\|u\|_{X^{s,b}_{\tau=0}} = \| \langle \xi \rangle^s \langle  \tau  \rangle^b \widehat{u}(\tau,\xi) \|_{L^2_{\tau \xi}} \, .$$
We also define $X^{s,b}_{\pm}[0,T]$ as the space of the restrictions of functions in $X^{s,b}_{\pm}$ to $[0,T] \times \mathbb{R}^2$ and similarly $X^{s,b}_{|\tau| = |\xi|}[0,T]$ and $X^{s,b}_{\tau =0}[0,T]$ . We frequently use the estimate $\|u\|_{X^{s,b}_{\pm}} \le \|u\|_{X^{s,b}_{|\tau|=|\xi|}}$ for $b \le 0$ and the reverse estimate for $b \ge 0$. 

We decompose $A=(A_1,A_2)$ into its divergence-free part $A^{df}$ and its 
"curl-\\-free" part $A^{cf}$:
$$ A=A^{df} + A^{cf} \, , $$
where
\begin{align*}
A^{df} &:= PA  := (- \Delta)^{-1} (\partial_2(\partial_1 A_2 - \partial_2 A_1),-\partial_1(\partial_1A_2-\partial_2 A_1)) \\
& = (R_2(R_1 A_2-R_2 A_1),-R_1(R_1 A_2 - R_2 A_1))\, ,\\
A^{cf} & := -(-\Delta)^{-1} \nabla \, div A = -(R_1(R_1A_1+R_2A_2),R_2(R_1A_1+R_2A_2)) \, ,
\end{align*}
and the Riesz-transform $R_j$ is defined by $R_j = |\nabla|^{-1} \partial_j$.

Then we obtain the equivalent system
\begin{align}
\label{1.11}
\partial_t A^{cf} &= - (-\Delta)^{-1} \nabla Im(\phi \overline{\partial_t \phi}) \\
\label{1.12}
\Box A^{df} & = -P( Im(\phi \overline{ \nabla \phi})  -A |\phi|^2) \\
\label{1.13}
\Box \phi & =- i(\partial^j A_j^{cf}) \phi - 2i A^{df}_j \partial^j \phi -2i A_j^{cf} \partial^j \phi + A^j A_j \phi + m^2 \phi \, .
\end{align}

Defining
\begin{align*}
\phi_{\pm} = \frac{1}{2}(\phi \pm i \langle \nabla \rangle^{-1} \partial_t \phi)&
 \Longleftrightarrow \phi=\phi_+ + \phi_- \, , \, \partial_t \phi = i \langle \nabla \rangle (\phi_+ - \phi_-) \\
 A^{df}_{\pm} = \frac{1}{2}(A^{df} \pm i \langle \nabla \rangle^{-1} \partial_t A^{df}) & \Longleftrightarrow A^{df} = A^{df}_+ + A_-^{df} \, , \, \partial_t A^{df} = i \langle \nabla \rangle(A^{df}_+ - A^{df}_-)
 \end{align*}
 we can rewrite (\ref{1.11}),(\ref{1.12}),(\ref{1.13}) as
 \begin{align}
 \label{1.11*}
 \partial_t A^{cf} &= - (-\Delta)^{-1} \nabla Im(\phi \overline{\partial_t \phi}) \\
 \label{1.12*}
(-i \partial_t \pm \langle \nabla \rangle)A_{\pm} ^{df} & = \mp 2^{-1} \langle \nabla \rangle^{-1} ( R.H.S. \, of \, (\ref{1.12}) - A^{df}) \\
\label{1.13*}
(-i \partial_t \pm \langle \nabla \rangle) \phi_{\pm} &= \mp 2^{-1} \langle \nabla \rangle^{-1}( R.H.S. \, of \, (\ref{1.13}) - \phi) \, .
\end{align}
The initial data are transformed as follows:
\begin{align}
\label{1.14*}
\phi_{\pm}(0) &= \frac{1}{2}(\phi(0) \pm i^{-1} \langle \nabla \rangle^{-1} (\partial_t \phi)(0)) \\
\label{1.15*}
A^{df}_\pm(0) & = \frac{1}{2}(A^{df}(0) \pm i^{-1} \langle \nabla \rangle^{-1} (\partial_t A^{df})(0)) \, .
\end{align}

Our main result is preferably formulated in terms of the system (\ref{1.11}),(\ref{1.12}),(\ref{1.13}).
\begin{theorem}
\label{Theorem}
1. Assume $r> \frac{1}{4}$ , $l\ge s>\max(\frac{1}{2}+\frac{l}{8},\frac{1}{4}+\frac{l}{2},\frac{1}{4}+\frac{r}{2},\frac{7}{16}+\frac{r}{4})$ , $r+\frac{1}{2} > s \ge r-\frac{1}{2}$ , $s > l-\frac{1}{2}$ and $\tilde{\epsilon} > 0$ sufficiently small. Let $\phi_0 \in H^s({\mathbb R}^2),$  $\phi_1 \in H^{s-1}({\mathbb R}^2),$  $a = a^{df} + a^{cf} \, ,$  $a' = a'^{df} + a'^{cf}$ 
be given with $a^{df} \in H^r({\mathbb R}^2)$ , $|\nabla|^{\tilde{\epsilon}} a^{cf} \in H^{l-\tilde{\epsilon}}({\mathbb R}^2)$ , $a'^{df} \in H^{r-1}({\mathbb R}^2)\,,$ 
which satisfy the compatability condition
\begin{equation}
\label{CC}
\partial^j a'_j = Im(\phi_0 \overline{\phi}_1) \, .
\end{equation}
Then there exists $T>0$, such that (\ref{1.11}),(\ref{1.12}),(\ref{1.13}) with initial conditions
$ \phi(0)= \phi_0$ , $(\partial_t \phi)(0) = \phi_1$ , $A^{df}(0) = a^{df}$ , $(\partial_t A^{df})(0)= a'^{df}$ , $A^{cf}(0) = a^{cf}$
has a unique local solution
$$ \phi= \phi_++ \phi_- \quad , \quad A=A^{df}_+ + A^{df}_- + A^{cf}$$
with 
$$ \phi_{\pm} \in X^{s,\frac{1}{2}+\epsilon}_{\pm}[0,T] \, , \, A^{df}_{\pm} \in X^{r,\frac{3}{4}+\epsilon}_{\pm}[0,T] \, , \, |\nabla|^{\tilde{\epsilon}}  A^{cf} \in X^{l-\tilde{\epsilon},\frac{1}{2}+\epsilon-}_{\tau =0}[0,T] \, ,$$
where $\epsilon >0$ is sufficiently small.\\
2. This solution satisfies
$$\phi_{\pm} \in C^0([0,T],H^s({\mathbb R}^2)) \, , \, A^{df}_{\pm} \in C^0([0,T],H^r({\mathbb R}^2)) \, , $$
$$|\nabla|^{\tilde{\epsilon}} A^{cf} \in C^0([0,T],H^{l-\tilde{\epsilon}}({\mathbb R}^2)) \, . $$
In the case $s=r=l=1$ the solution is (unconditionally) unique in these spaces.
\end{theorem}
{\bf Remarks:}
\begin{itemize}
\item
The compatability conditon (\ref{CC}), which is necessary in view of (\ref{1.3}), determines $a'^{cf}$ as $a'^{cf} = -(-\Delta)^{-1} \nabla(Im(\phi_0 \overline{\phi}_1))$ . \\
Is is not difficult to see that $a'^{cf}$ fulfills $|\nabla|^{\tilde{\epsilon}} a'^{cf} \in H^{l-1-\tilde{\epsilon}}({\mathbb R}^2)$. One only has to show that
$$ \||\nabla|^{-1+\tilde{\epsilon}}(\phi_0 \overline{\phi}_1) \|_{H^{l-1-\tilde{\epsilon}}} \lesssim \|\phi_0\|_{H^s} \|\phi_1\|_{H^{s-1}} \ . $$
By duality this is equivalent to
$$\|\phi_0 \phi_2 \|_{H^{1-s}} \lesssim \|\phi_0\|_{H^s} \| |\nabla|^{1-\tilde{\epsilon}} \phi_2\|_{H^{1-l+\tilde{\epsilon}}} \, . $$
In the case of high frequencies of $\phi_2$ this follows from the Sobolev multiplication law (\ref{SML}) using $2s-l >0$ , and the low frequency case can be easily handled using $s>\frac{1}{2}$.  
\item The minimal regularity assumptions are given by $r=\frac{1}{4}+$ , $l=s=\frac{1}{2} + \frac{1}{14}+.$   
\item If one wants to have the same regularity for $\phi$ and $A$ one also checks that $r=l=s=\frac{1}{2}+\frac{1}{12}+$ is admissible. 
\item The choice $r=l=s=1$ is of course admissible. 
\end{itemize}

Fundamental for us are the following estimates. We frequently use the classical Sobolev multiplication law in dimension two :
\begin{equation}
\label{SML}
\|uv\|_{H^{-s_0}} \lesssim \|u\|_{H^{s_1}} \|v\|_{H^{s_2}} \, ,
\end{equation}
if $s_0 + s_1+s_2 \ge 1$ and $s_0+s_1+s_2 \ge \max(s_0,s_1,s_2)$ , where at most one of these inequalities is an equality.

The corresponding  bilinear estimates in wave-Sobolev spaces were proven by d'Ancona, Foschi and Selberg in the two-dimensional case in \cite{AFS} in a form which includes some more limit cases which we do not need.
\begin{prop}
\label{Prop.2}
For $s_0,s_1,s_2,b_0,b_1,b_2 \in {\mathbb R}$ and $u,v \in   {\mathcal S} ({\mathbb R}^{2+1})$ the estimate
$$\|uv\|_{X_{|\tau|=|\xi|}^{-s_0,-b_0}} \lesssim \|u\|_{X^{s_1,b_1}_{|\tau|=|\xi|}} \|v\|_{X^{s_2,b_2}_{|\tau|=|\xi|}} $$ 
holds, provided the following conditions are satisfied:
\begin{align*}
\nonumber
& b_0 + b_1 + b_2 > \frac{1}{2} \, ,
& b_0 + b_1 \ge 0 \, ,\quad \qquad  
& b_0 + b_2 \ge 0 \, ,
& b_1 + b_2 \ge 0
\end{align*}
\begin{align*}
\nonumber
&s_0+s_1+s_2 > \frac{3}{2} -(b_0+b_1+b_2) \\
\nonumber
&s_0+s_1+s_2 > 1 -\min(b_0+b_1,b_0+b_2,b_1+b_2) \\
\nonumber
&s_0+s_1+s_2 > \frac{1}{2} - \min(b_0,b_1,b_2) \\
\nonumber
&s_0+s_1+s_2 > \frac{3}{4} \\
 &(s_0 + b_0) +2s_1 + 2s_2 > 1 \\
\nonumber
&2s_0+(s_1+b_1)+2s_2 > 1 \\
\nonumber
&2s_0+2s_1+(s_2+b_2) > 1 
\end{align*}
\begin{align*}
\nonumber
&s_1 + s_2 \ge \max(0,-b_0) \, ,\quad
\nonumber
s_0 + s_2 \ge \max(0,-b_1) \, ,\quad
\nonumber
s_0 + s_1 \ge \max(0,-b_2)   \, .
\end{align*}
\end{prop}
Moreover we need the standard Strichartz estimate combined with the transfer principle (for a proof see \cite{S1}, Theorem 8):
\begin{equation}
\label{Str}
\|u\|_{L^6_{xt}} \lesssim \|u\|_{X^{\frac{1}{2},\frac{1}{2}+}_{|\tau|=|\xi|}} 
\end{equation}
and the following estimate, which essentially goes back to Tataru \cite{KMBT}.
\begin{lemma}
\label{Lemma}
For $2 \le p \le 6$ the following estimates hold:
\begin{align*}
\|u\|_{L^p_x L^2_t} & \lesssim \|u\|_{X^{\frac{1}{2}(\frac{1}{2}-\frac{1}{p}),\frac{3}{2}(\frac{1}{2}-\frac{1}{p})+}_{|\tau|=|\xi|}} \, , \\
\|u\|_{L^p_x L^{2+}_t} & \lesssim \|u\|_{X^{\frac{1}{2}(\frac{1}{2}-\frac{1}{p})+,\frac{3}{2}(\frac{1}{2}-\frac{1}{p})+}_{|\tau|=|\xi|}} \, .
\end{align*}
\end{lemma}
\begin{proof}
By \cite{KMBT}, Thm. B2 we obtain
$ \|{\mathcal F}_t u \|_{L^2_{\tau} L^6_x} \lesssim \|u_0\|_{\dot{H}^{\frac{1}{6}}} \, , $
if $u= e^{it|\nabla|} u_0$ and ${\mathcal F}_t$ denotes the Fourier transform with respect to time. This implies by Plancherel and Minkowski's inequality
$$ \|u\|_{L^6_x L^2_t} = \|{\mathcal F}_t u \|_{L^6_x L^2_{\tau}} \le \|{\mathcal F}_t u \|_{L^2_{\tau} L^6_x} \lesssim \|u_0\|_{\dot{H}^{\frac{1}{6}}} \, . $$
The transfer principle \cite{S1}, Prop. 8 implies 
\begin{equation}
\|u\|_{L^6_x L^2_t} \lesssim \|u\|_{X^{\frac{1}{6},\frac{1}{2}+}_{|\tau|=|\xi|}} \, .
\end{equation}
Interpolation with (\ref{Str}) gives
\begin{equation}
\|u\|_{L^6_x L^{2+}_t} \lesssim \|u\|_{X^{\frac{1}{6}+,\frac{1}{2}+}_{|\tau|=|\xi|}} \, .
\end{equation}
Interpolation of the last two inequalities with the trivial identity
$\|u\|_{L^2_x L^2_t} = \|u\|_{X^{0,0}_{|\tau|=|\xi|}} $ completes the proof.
\end{proof}

\section{Proof of the Theorem}
We now consider the Cauchy problem (\ref{1.11*}),(\ref{1.12*}),(\ref{1.13*}),(\ref{1.14*}),(\ref{1.15*}). 
Klainerman and Machedon detected that $A^{df} \cdot \nabla \phi$ and $P(\phi \overline{\nabla \phi})_k$ are null forms. An elementary calculation namely shows that
\begin{align}
\label{1.13'}
 A^{df}_i \partial^i \phi & = Q_{12}(\phi,|\nabla|^{-1}(R_1 A_2- R_2 A_1)) \end{align}
 and
 \begin{align}
\label{1.12'}
P(\phi \overline{\nabla \phi})_1 & = -2i R_2 |\nabla|^{-1} Q_{12}(Re \phi, Im \phi) \\
\label{1.12''}
P(\phi \overline{\nabla \phi})_2 & = 2i R_1 |\nabla|^{-1} Q_{12}(Re \phi, Im \phi)  \, , 
\end{align}
where the null form $Q_{12}$ is defined by
$$ Q_{12}(u,v):= \partial_1 u \partial_2 v - \partial_1 u \partial_2 v \, .$$
In order to estimate these null forms we also use the following estimate for the angle $\angle(\xi_1,\xi_2)$ between two vectors $\xi_1$ and $\xi_2$.
\begin{lemma}
Assume $0 \le \alpha,\beta,\gamma \le \frac{1}{2}$ and $\xi_i \in {\mathbb R}^2$ , $\tau_i \in{\mathbb R}$ $(i=1,2,3)$ with $\xi_1+\xi_2+\xi_3 = 0$ , $\tau_1+\tau_2+\tau_3 =0$ . Then the following estimate holds for independent signs $\pm$ and $\pm'$ :
\begin{equation}
\label{31}
\angle(\pm \xi_1,\pm' \xi_2) \lesssim \Big( \frac{\langle -\tau_1 \pm |\xi_1| \rangle}{\min(\langle \xi_1 \rangle,\langle \xi_2 \rangle)} \Big)^{\alpha} + \Big( \frac{\langle -\tau_2 \pm' |\xi_2| \rangle}{\min(\langle \xi_1 \rangle,\langle \xi_2 \rangle)} \Big)^{\beta} + \Big( \frac{\langle |\tau_3| - |\xi_3| \rangle}{\min(\langle \xi_1 \rangle,\langle \xi_2 \rangle)} \Big)^{\gamma} \, .
\end{equation}
\end{lemma}
For a proof see for example \cite{S}, Lemma 2.1.

\begin{proof}[{\bf Proof of Theorem \ref{Theorem}}]
{\bf Proof of part 1:} 
We use (\ref{1.13'}),(\ref{1.12'}),(\ref{1.12''}). By a contraction argument the local existence and uniqueness proof is reduced to suitable multilinear estimates for the right hand sides of (\ref{1.11*}),(\ref{1.12*}),(\ref{1.13*}). For (\ref{1.12*}), e.g. , we make use of the following well-known estimate for a solution of the linear equation
$
(-i \partial_t \pm \langle \nabla \rangle)A_{\pm} ^{df}  = G$ , namely
$$
\|A^{df}_{\pm}\|_{X^{k,b}_{\pm}[0,T]} \lesssim \|A^{df}_{\pm}(0)\|_{H^k} + T^{b'-b} \| G \|_{X^{k,b'-1}_{\pm}[0,T]} \, , $$
which holds for $k\in{\mathbb R}$  , $\frac{1}{2} < b \le b' < 1$ and $0<T \le 1$ .

Thus the local existence and uniqueness for large data (in which case we have to choose $b < b'$) , in the regularity class $$ \phi_{\pm} \in X^{s,\frac{1}{2}+\epsilon}_{\pm}[0,T] \, , \, A^{df}_{\pm} \in X^{r,\frac{3}{4}+\epsilon}_{\pm}[0,T] \, , \, |\nabla|^{\tilde{\epsilon}}  A^{cf} \in X^{l-\tilde{\epsilon},\frac{1}{2}+\epsilon-}_{\tau =0}[0,T] $$can be reduced to 
the following estimates for independent signs $\pm $ , $\pm' $ , $\pm''$ :
\begin{align}
\label{27}
\| |\nabla|^{-1+\tilde{\epsilon}} (\phi_1 \partial_t \phi_2)\|_{X^{l-\tilde{\epsilon},-\frac{1}{2}+\epsilon}_{\tau=0}} &\lesssim \|\phi_1\|_{X^{s,\frac{1}{2}+\epsilon}_{|\tau|=|\xi|}} \|\phi_2\|_{X^{s,\frac{1}{2}+\epsilon}_{|\tau|=|\xi|}} \, ,
\\
\label{28}
\||\nabla|^{-1}Q_{ij}(\phi_1,\phi_2)\|_{X^{r-1,-\frac{1}{4}+2\epsilon}_{\pm''}} 
&\lesssim \|\phi_1\|_{X^{s,\frac{1}{2}+\epsilon}_{\pm}} 
\|\phi_2\|_{X^{s,\frac{1}{2}+\epsilon}_{\pm'}} \, , \\
\label{28'}
\|Q_{ij}(|\nabla|^{-1}\phi_1,\phi_2)\|_{X^{s-1,-\frac{1}{2}+2\epsilon}_{\pm''}} & \lesssim \|\phi_1\|_{X^{r,\frac{3}{4}+\epsilon}_{\pm}} \|\phi_2\|_{X^{s,\frac{1}{2}+\epsilon}_{\pm'}} \, , \\
\label{29}
\| \nabla A \phi \|_{X^{s-1,-\frac{1}{2}+2\epsilon}_{|\tau|=|\xi|}} +
\| A \nabla \phi \|_{X^{s-1,-\frac{1}{2}+2\epsilon}_{|\tau|=|\xi|}} &\lesssim \||\nabla|^{\tilde{\epsilon}} A\|_{X^{l-\tilde{\epsilon},\frac{1}{2}+\epsilon-}_{\tau =0}}  \|\phi\|_{X^{s,\frac{1}{2}+\epsilon}_{|\tau|=|\xi|}} \, ,
\end{align}
\begin{align}
\label{30}
\| A \phi_1 \phi_2 \|_{X^{r-1,-\frac{1}{4}+2\epsilon}_{|\tau|=|\xi|}} &\lesssim  \min(\|A\|_{X^{r,\frac{3}{4}+\epsilon}_{|\tau|=|\xi|}}, \||\nabla|^{\tilde{\epsilon}}A\|_{X^{l-\tilde{\epsilon},\frac{1}{2}+\epsilon-}_{\tau =0}} ) \prod_{i=1}^2 \|\phi_i\|_{X^{s,\frac{1}{2}+\epsilon}_{|\tau|=|\xi|}} \, ,
\\
\label{30'}
\| A_1 A_2 \phi \|_{X^{s-1,-\frac{1}{2}+2\epsilon}_{|\tau|=|\xi|}} &\lesssim \prod_{i=1}^2 \min(\|A_i\|_{X^{r,\frac{3}{4}+\epsilon}_{|\tau|=|\xi|}},\||\nabla|^{\tilde{\epsilon}} A_i\|_{X^{l-\tilde{\epsilon},\frac{1}{2}+\epsilon-}_{\tau =0}} ) \|\phi\|_{X^{s,\frac{1}{2}+\epsilon}_{|\tau|=|\xi|}} \, .
\end{align}
{\bf Proof of (\ref{28'}):} The Fourier multiplier of $Q_{12}(|\nabla|^{-1} \phi_1,\phi_2)$ is bounded by
\begin{equation}
\label{32}
\frac{|\xi_1 \times \xi_2|}{|\xi_1|} \lesssim |\xi_2| \angle(\pm\xi_1,\pm' \xi_2) \, , 
\end{equation}
where $\xi_1 \times \xi_2 := \xi_{11} \xi_{22}-\xi_{21} \xi_{12}$ . If $\xi_1+\xi_2+\xi_3 =0$ we also have
\begin{equation}
\label{33}
\frac{|\xi_1 \times \xi_2|}{|\xi_1|} = \frac{|\xi_1 \times \xi_3|}{|\xi_1|} \lesssim |\xi_3|\angle(\pm\xi_1,\pm'' \xi_3) \, . 
\end{equation}
{\bf 1.} In the case $|\xi_3| \gtrsim \max(|\xi_1|,|\xi_2|)$ we use (\ref{32}). It suffices to show
\begin{align}
\nonumber
&\int_* \frac{\widehat{u_1}(\xi_1,\tau_1)}{\langle \xi_1 \rangle^r \langle - \tau_1 \pm |\xi_1| \rangle^{\frac{3}{4}+}} \frac{\widehat{u_2}(\xi_2,\tau_2) |\xi_2|}{\langle \xi_2 \rangle^s \langle - \tau_2 \pm' |\xi_2| \rangle^{\frac{1}{2}+}} 
\frac{\widehat{u_3}(\xi_3,\tau_3)}{\langle \xi_3 \rangle^{1-s} \langle      |\tau_3| - |\xi_3| \rangle^{\frac{1}{2}-}} \cdot\\
\label{34}
& \hspace{17em} \cdot\angle(\pm \xi_1, \pm' \xi_2) \, d\xi d\tau \lesssim \prod_{i=1}^3 \|u_i\|_{L^2_{xt}} \, .
\end{align}
The Fourier transforms are nonnegative without loss of generality.
Here * denotes integration over $\sum_{i=1}^3 \xi_i =0$ , $\sum_{i=1}^3 \tau_i =0$ and $d\xi d\tau = d\xi_1 d\xi_2 d\xi_3 d\tau_1 d \tau_2 d\tau_3$ . \\
We use (\ref{31}) with $\alpha = \beta = \frac{1}{2}$ , $\gamma = \frac{1}{2}-$ . \\
{\bf 1.1.} $|\xi_1| \le |\xi_2|$ . If the first term on the r.h.s. of (\ref{31}) is dominant we use $|\xi_3| \sim |\xi_2|$ and reduce to
\begin{align*}
\nonumber
&\int_* \frac{\widehat{u_1}(\xi_1,\tau_1)}{\langle \xi_1 \rangle^{r+\frac{1}{2}} \langle | \tau_1| - |\xi_1| \rangle^{\frac{1}{4}+}} 
\frac{\widehat{u_2}(\xi_2,\tau_2)}{ \langle |\tau_2| - |\xi_2| \rangle^{\frac{1}{2}+}} 
\frac{\widehat{u_3}(\xi_3,\tau_3)}{ \langle      |\tau_3| - |\xi_3| \rangle^{\frac{1}{2}-}} \,
d\xi d\tau 
\lesssim \prod_{i=1}^3 \|u_i\|_{L^2_{xt}} \, ,
\end{align*}
which follows from Prop. \ref{Prop.2} for $r > \frac{1}{4}$ , where we need the factor $\langle |\tau_1|-|\xi_1|\rangle^{\frac{1}{4}+}$ in the denominator. For the second and third term on the r.h.s. of (\ref{31}) we only have to show
\begin{align*}
\nonumber
&\int_* \frac{\widehat{u_1}(\xi_1,\tau_1)}{\langle \xi_1 \rangle^{r+\frac{1}{2}} \langle |\tau_1| - |\xi_1| \rangle^{\frac{3}{4}+}} 
\widehat{u_2}(\xi_2,\tau_2) 
\frac{\widehat{u_3}(\xi_3,\tau_3)}{ \langle      |\tau_3| - |\xi_3| \rangle^{\frac{1}{2}-}} \,
d\xi d\tau 
\lesssim \prod_{i=1}^3 \|u_i\|_{L^2_{xt}} \, ,
\end{align*}
and
\begin{align*}
\nonumber
&\int_* \frac{\widehat{u_1}(\xi_1,\tau_1)}{\langle \xi_1 \rangle^{r+\frac{1}{2}-} \langle | \tau_1 | - |\xi_1| \rangle^{\frac{3}{4}+}} 
\frac{\widehat{u_2}(\xi_2,\tau_2)}{ \langle |\tau_2| - |\xi_2| \rangle^{\frac{1}{2}+}} 
\widehat{u_3}(\xi_3,\tau_3) \,
d\xi d\tau 
\lesssim \prod_{i=1}^3 \|u_i\|_{L^2_{xt}} \, ,
\end{align*}
respectively, both of which follow from Prop. \ref{Prop.2} for $r>\frac{1}{4}$ . \\
{\bf 1.2.} $|\xi_1| \ge |\xi_2|$ . Using $|\xi_3| \sim |\xi_1|$ the l.h.s. of (\ref{34}) is bounded by 
\begin{align*}
\nonumber
&\int_* \frac{\widehat{u_1}(\xi_1,\tau_1)}{\langle \xi_1 \rangle^{1-s+r} \langle | \tau_1| - |\xi_1| \rangle^{\frac{1}{4}+}} \frac{\widehat{u_2}(\xi_2,\tau_2) |\xi_2|}{\langle \xi_2 \rangle^{s-\frac{1}{2}} \langle | \tau_2| - |\xi_2| \rangle^{\frac{1}{2}+}} 
\frac{\widehat{u_3}(\xi_3,\tau_3)}{ \langle |\tau_3| - |\xi_3| \rangle^{\frac{1}{2}-}} \, d\xi d\tau
\end{align*}
for the first term on the r.h.s. of (\ref{31}). Similarly for the second and third term we obtain the bounds
\begin{align*}
\nonumber
&\int_* \frac{\widehat{u_1}(\xi_1,\tau_1)}{\langle \xi_1 \rangle^{1-s+r} \langle | \tau_1| - |\xi_1| \rangle^{\frac{3}{4}+}} \frac{\widehat{u_2}(\xi_2,\tau_2) |\xi_2|}{\langle \xi_2 \rangle^{s-\frac{1}{2}} } 
\frac{\widehat{u_3}(\xi_3,\tau_3)}{ \langle |\tau_3| - |\xi_3| \rangle^{\frac{1}{2}-}} \, d\xi d\tau
\end{align*}
and
\begin{align*}
\nonumber
&\int_* \frac{\widehat{u_1}(\xi_1,\tau_1)}{\langle \xi_1 \rangle^{1-s+r} \langle | \tau_1| - |\xi_1| \rangle^{\frac{3}{4}+}} \frac{\widehat{u_2}(\xi_2,\tau_2) |\xi_2|}{\langle \xi_2 \rangle^{s-\frac{1}{2}-}\langle | \tau_2| - |\xi_2| \rangle^{\frac{1}{2}+}} 
\widehat{u_3}(\xi_3,\tau_3) \, d\xi d\tau \, ,
\end{align*}
respectively, all of which are bounded by $\prod_{i=1}^3 \|u_i\|_{L^2_{xt}}$ for $r>\frac{1}{4}$  , $s > \frac{1}{2}$ and $s \le r+1$ by Prop. \ref{Prop.2}.\\
{\bf 2.} In the case $|\xi_3| \ll |\xi_1| \sim |\xi_2|$ we use (\ref{33}).  It suffices to show
\begin{align}
\nonumber
&\int_* \frac{\widehat{u_1}(\xi_1,\tau_1)}{\langle \xi_1 \rangle^r \langle - \tau_1 \pm |\xi_1| \rangle^{\frac{3}{4}+}} \frac{\widehat{u_2}(\xi_2,\tau_2)}{\langle \xi_2 \rangle^s \langle |\tau_2| - |\xi_2| \rangle^{\frac{1}{2}+}} 
\frac{\widehat{u_3}(\xi_3,\tau_3) |\xi_3|}{\langle \xi_3 \rangle^{1-s} \langle      -\tau_3 \pm'' |\xi_3| \rangle^{\frac{1}{2}-}} \cdot\\
\label{35}
& \hspace{17em} \cdot\angle(\pm \xi_1, \pm'' \xi_3) \, d\xi d\tau \lesssim \prod_{i=1}^3 \|u_i\|_{L^2_{xt}} \, .
\end{align}
Using $|\xi_2| \gtrsim |\xi_3|$ and (\ref{31}) with $\alpha=\beta=\frac{1}{2}$ , $\gamma=\frac{1}{2}-$ and $\xi_2$ permuted with $\xi_3$ we bound the l.h.s. of (\ref{35}) by
\begin{align*}
\nonumber
&\int_* \frac{\widehat{u_1}(\xi_1,\tau_1)}{\langle \xi_1 \rangle^r \langle | \tau_1| - |\xi_1| \rangle^{\frac{3}{4}+}} \widehat{u_2}(\xi_2,\tau_2)  
\frac{\widehat{u_3}(\xi_3,\tau_3)}{\langle \xi_3 \rangle^{\frac{1}{2}} \langle |\tau_3| - |\xi_3| \rangle^{\frac{1}{2}-}} \, d\xi d\tau
\\
\nonumber
&+\int_* \frac{\widehat{u_1}(\xi_1,\tau_1)}{\langle \xi_1 \rangle^r \langle | \tau_1| - |\xi_1| \rangle^{\frac{3}{4}+}} \frac{\widehat{u_2}(\xi_2,\tau_2)}{ \langle | \tau_2| - |\xi_2| \rangle^{\frac{1}{2}+}} 
\frac{\widehat{u_3}(\xi_3,\tau_3)}{ \langle \xi_3 \rangle^{\frac{1}{2}-}} \, d\xi d\tau
\\
\nonumber
&+\int_* \frac{\widehat{u_1}(\xi_1,\tau_1)}{\langle \xi_1 \rangle^r \langle | \tau_1| - |\xi_1| \rangle^{\frac{1}{4}+}} \frac{\widehat{u_2}(\xi_2,\tau_2)}{ \langle | \tau_2| - |\xi_2| \rangle^{\frac{1}{2}+}} 
\frac{\widehat{u_3}(\xi_3,\tau_3)}{ \langle \xi_3 \rangle^{\frac{1}{2}} \langle |\tau_3| - |\xi_3| \rangle^{\frac{1}{2}-}} \, d\xi d\tau \, ,
\end{align*}
which gives (\ref{35}) by Prop. \ref{Prop.2} and completes the proof of (\ref{28'}). \\
{\bf Proof of (\ref{28}):} We recall (\ref{32}) and (\ref{33}) and obtain the following bounds for the Fourier multiplier of $Q_{12}(\phi_1,\phi_2)$ :
\begin{align}
\label{36}
|\xi_1 \times \xi_2| & \lesssim |\xi_1| |\xi_2| \angle(\pm \xi_1,\pm' \xi_2) \, , \\
\label{37}
|\xi_1 \times \xi_2| & \lesssim |\xi_1| |\xi_3| \angle(\pm \xi_1,\pm'' \xi_3)  \, .
\end{align}
{\bf 1.} In the case $|\xi_3| \gtrsim \max(|\xi_1|,|\xi_2|)$ we use (\ref{36}) and reduce the desired estimate to
\begin{align}
\nonumber
&\int_* \frac{\widehat{u_1}(\xi_1,\tau_1) |\xi_1|}{\langle \xi_1 \rangle^s \langle - \tau_1 \pm |\xi_1| \rangle^{\frac{1}{2}+}} \frac{\widehat{u_2}(\xi_2,\tau_2) |\xi_2|}{\langle \xi_2 \rangle^s \langle - \tau_2 \pm' |\xi_2| \rangle^{\frac{1}{2}+}} 
\frac{\widehat{u_3}(\xi_3,\tau_3)}{|\xi_3| \langle \xi_3 \rangle^{1-r} \langle      |\tau_3| - |\xi_3| \rangle^{\frac{1}{4}-}} \cdot\\
\label{38}
& \hspace{17em} \cdot\angle(\pm \xi_1, \pm' \xi_2) \, d\xi d\tau \lesssim \prod_{i=1}^3 \|u_i\|_{L^2_{xt}} \, .
\end{align}
By symmetry we may assume $|\xi_1| \le |\xi_2|$ . We estimate $\angle(\pm\xi_1,\pm' \xi_2)$ by (\ref{31}) with $\alpha = \beta = \frac{1}{2}$ , $\gamma = \frac{1}{4}-$ . We estimate the l.h.s. of (\ref{38}) concerning the first term on the r.h.s. of (\ref{31}) using $|\xi_3| \sim |\xi_2|$ by
\begin{align*}
&\int_* \frac{\widehat{u_1}(\xi_1,\tau_1)}{\langle \xi_1 \rangle^{s-\frac{1}{2}}} \frac{\widehat{u_2}(\xi_2,\tau_2)}{\langle \xi_2 \rangle^{s-r+1} \langle | \tau_2| - |\xi_2| \rangle^{\frac{1}{2}+}} 
\frac{\widehat{u_3}(\xi_3,\tau_3)}{ \langle      |\tau_3| - |\xi_3| \rangle^{\frac{1}{4}-}}\, d\xi d\tau \, ,
\end{align*}
which gives (\ref{38}) by Prop. \ref{Prop.2} , where we used $s> \frac{1}{4} + \frac{r}{2}$ and $s \ge r-\frac{1}{2}$ . For the second term on the r.h.s. of (\ref{31}) we control the l.h.s. of (\ref{38}) by
\begin{align*}
&\int_* \frac{\widehat{u_1}(\xi_1,\tau_1)}{\langle \xi_1 \rangle^{s-\frac{1}{2}} \langle | \tau_1| - |\xi_1| \rangle^{\frac{1}{2}+}} \frac{\widehat{u_2}(\xi_2,\tau_2)}{\langle \xi_2 \rangle^{s+1-r}} 
\frac{\widehat{u_3}(\xi_3,\tau_3)}{ \langle      |\tau_3| - |\xi_3| \rangle^{\frac{1}{4}-}} \,d\xi d\tau \\
& \lesssim \int_* \frac{\widehat{u_1}(\xi_1,\tau_1)}{\langle \xi_1 \rangle^{2s+\frac{1}{2}-r} \langle | \tau_1| - |\xi_1| \rangle^{\frac{1}{2}+}} \widehat{u_2}(\xi_2,\tau_2)
\frac{\widehat{u_3}(\xi_3,\tau_3)}{ \langle      |\tau_3| - |\xi_3| \rangle^{\frac{1}{4}-}} \,d\xi d\tau \, .
\end{align*}
We apply Prop. \ref{Prop.2} using $s>\frac{1}{4}+\frac{r}{2}$ and $s \ge r-1$ to obtain (\ref{38}). For the last term on the r.h.s. of (\ref{31}) we estimate the l.h.s. of (\ref{38}) using $|\xi_3| \sim |\xi_2| \gtrsim |\xi_1|$ and $s \ge r-1$ as follows:
\begin{align*}
&\int_* \frac{\widehat{u_1}(\xi_1,\tau_1)}{\langle \xi_1 \rangle^{s-\frac{3}{4}-} \langle | \tau_1| - |\xi_1| \rangle^{\frac{1}{2}+}} 
\frac{\widehat{u_2}(\xi_2,\tau_2)}{\langle \xi_2 \rangle^{s} \langle      |\tau_2| - |\xi_2| \rangle^{\frac{1}{2}+}} \frac{\widehat{u_3}   (\xi_3,\tau_3)}{\langle \xi_3 \rangle^{1-r}}  \,d\xi d\tau \\
& \lesssim \int_* \frac{\widehat{u_1}(\xi_1,\tau_1)}{\langle \xi_1 \rangle^{2s-r+\frac{1}{4}-} \langle | \tau_1| - |\xi_1| \rangle^{\frac{1}{2}+}} 
\frac{\widehat{u_2}(\xi_2,\tau_2)}{ \langle |\tau_2| - |\xi_2| \rangle^{\frac{1}{2}+}} \widehat{u_3}   (\xi_3,\tau_3)  \,d\xi d\tau \lesssim \prod_{i=1}^3 \|u_i\|_{L^2_{xt}} \, .
\end{align*}
The last estimate follows from Prop. \ref{Prop.2} using $s>\frac{1}{4}+\frac{r}{2}$ again. \\
{\bf 2.} In the case $|\xi_3| \ll |\xi_1|\sim|\xi_2|$ we use (\ref{37}) and reduce the desired estimate to
\begin{align}
\nonumber
&\int_* \frac{\widehat{u_1}(\xi_1,\tau_1) |\xi_1|}{\langle \xi_1 \rangle^s \langle - \tau_1 \pm |\xi_1| \rangle^{\frac{1}{2}+}} \frac{\widehat{u_2}(\xi_2,\tau_2)}{\langle \xi_2 \rangle^s \langle |\tau_2| - |\xi_2| \rangle^{\frac{1}{2}+}} 
\frac{\widehat{u_3}(\xi_3,\tau_3)}{\langle \xi_3 \rangle^{1-r} \langle -\tau_3 \pm'' |\xi_3| \rangle^{\frac{1}{4}-}} \cdot\\
\label{39}
& \hspace{17em} \cdot\angle(\pm \xi_1, \pm'' \xi_3) \, d\xi d\tau \lesssim \prod_{i=1}^3 \|u_i\|_{L^2_{xt}} \, .
\end{align}
We estimate $\angle(\pm\xi_1,\pm'' \xi_3)$ by (\ref{31}) with $\alpha = \beta = \frac{1}{2}$ , $\gamma = \frac{1}{4}-$ . We bound the l.h.s. of (\ref{39}) for the first term on the r.h.s. of (\ref{31}) (and similarly for the second term) using $|\xi_1| \sim |\xi_2|$ and $ s \ge \frac{1}{2}$ by
\begin{align*}
&\int_* \frac{\widehat{u_1}(\xi_1,\tau_1)}{\langle \xi_1 \rangle^{s-\frac{1}{2}}} \frac{\widehat{u_2}(\xi_2,\tau_2)}{\langle \xi_2 \rangle^{s-\frac{1}{2}} \langle | \tau_2| - |\xi_2| \rangle^{\frac{1}{2}+}} 
\frac{\widehat{u_3}(\xi_3,\tau_3)}{ \langle \xi_3 \rangle^{\frac{3}{2}-r} \langle |\tau_3| - |\xi_3| \rangle^{\frac{1}{4}-}}\, d\xi d\tau \\
&\lesssim \int_* \widehat{u_1}(\xi_1,\tau_1) \frac{\widehat{u_2}(\xi_2,\tau_2)}{ \langle | \tau_2| - |\xi_2| \rangle^{\frac{1}{2}+}} 
\frac{\widehat{u_3}(\xi_3,\tau_3)}{\langle \xi_3 \rangle^{2s-r+\frac{1}{2}}  \langle |\tau_3| - |\xi_3| \rangle^{\frac{1}{4}-}}\, d\xi d\tau \, ,
\end{align*}
which again implies (\ref{39}) by Prop. \ref{Prop.2} using $s> \frac{1}{4} + \frac{r}{2}$ . \\
For the last term on the r.h.s. of (\ref{31}) we estimate the l.h.s. of (\ref{39}) by
\begin{align*}
&\int_* \frac{\widehat{u_1}(\xi_1,\tau_1)}{\langle \xi_1 \rangle^{s-\frac{1}{2}} \langle | \tau_1| - |\xi_1| \rangle^{\frac{1}{2}+}} 
\frac{\widehat{u_2}(\xi_2,\tau_2)}{\langle \xi_2 \rangle^{s-\frac{1}{2}} \langle      |\tau_2| - |\xi_2| \rangle^{\frac{1}{2}+}} \frac{\widehat{u_3}   (\xi_3,\tau_3)}{\langle \xi_3 \rangle^{1-r+\frac{1}{4}-}}  \,d\xi d\tau \, .
\end{align*}
If $r < \frac{5}{4}$ we apply Prop. \ref{Prop.2} using $s>\frac{1}{4}+\frac{r}{2}$ , $s> r-\frac{1}{2}$ , $ s > \frac{1}{2}$ and $s>\frac{7}{16} + \frac{r}{4}$ , which implies (\ref{39}). \\
If $r \ge \frac{5}{4}$ we use $|\xi_3| \ll |\xi_1| \sim |\xi_2|$ and obtain the bound
\begin{align*}
&\int_* \frac{\widehat{u_1}(\xi_1,\tau_1)}{\langle \xi_1 \rangle^{s-\frac{r}{2}+\frac{1}{8}-}  \langle |\tau_1| - |\xi_1| \rangle^{\frac{1}{2}+}} \frac{\widehat{u_2}(\xi_2,\tau_2)}{\langle \xi_2 \rangle^{s-\frac{r}{2}+\frac{1}{8}-} \langle | \tau_2| - |\xi_2| \rangle^{\frac{1}{2}+}} 
\widehat{u_3}(\xi_3,\tau_3)\, d\xi d\tau \, ,
\end{align*}
which again implies (\ref{39}) by Prop. \ref{Prop.2} using $s> \frac{1}{4} + \frac{r}{2}$ , completing the proof of (\ref{28}). \\
{\bf Proof of (\ref{27}):} 
We first remark that the singularity of $|\nabla|^{-1+\tilde{\epsilon}}$ ($\tilde{\epsilon} > 0$) is harmless in two dimensions (\cite{T}, Cor. 8.2) and it can be replaced by $\langle \nabla \rangle^{-1+\tilde{\epsilon}}$. As a first step we use Sobolev's multiplikation law (\ref{SML}) and obtain
\begin{align*}
\big|\int \int u_1 u_2 u_3 dx dt\big| \lesssim \|u_1\|_{X^{s,\frac{1}{2}+\epsilon}_{\tau =0}}
\|u_2\|_{X^{s,-\frac{1}{2}+\epsilon}_{\tau =0}}
\|u_3\|_{X^{1-l,\frac{1}{2}-\epsilon}_{\tau =0}} 
\end{align*}
under the assumptions $s > \frac{l}{2}$ and $ s > l-1  $ . This implies taking the time derivative into account 
\begin{align}
\label{40}
&\| \langle\nabla \rangle^{-1+\tilde{\epsilon}} (\phi_1 \partial_t \phi_2)\|_{X^{l-\tilde{\epsilon},-\frac{1}{2}+\epsilon}_{\tau=0}} \lesssim \|\phi_1\|_{X^{s,\frac{1}{2}+\epsilon}_{\tau=0}} \|\phi_2\|_{X^{s,\frac{1}{2}+\epsilon}_{\tau=0}} \, .
\end{align}
In a second step we want to prove
\begin{align}
\nonumber
&\| \langle\nabla \rangle^{-1+\tilde{\epsilon}} (\phi_1 \partial_t \phi_2)\|_{X^{l-\tilde{\epsilon},-\frac{1}{2}+\epsilon}_{\tau=0}} + \| \langle \nabla \rangle^{-1+\tilde{\epsilon}} (\phi_2 \partial_t \phi_1)\|_{X^{l-\tilde{\epsilon},-\frac{1}{2}+\epsilon}_{\tau=0}} \\
\label{41}
& \hspace{20em}\lesssim \|\phi_1\|_{X^{s,\frac{1}{2}+\epsilon}_{|\tau|=|\xi|}} \|\phi_2\|_{X^{s,\frac{1}{2}+\epsilon}_{\tau=0}} \, .
\end{align}
If $\widehat{\phi_1}(\xi_3,\tau_3)$ is supported in  $ ||\tau_3|-|\xi_3|| \gtrsim |\xi_3| $ we have the trivial bound
\begin{equation}
\label{41'}
 \|\phi_1\|_{X^{s,\frac{1}{2}+\epsilon}_{\tau=0}} \lesssim \|\phi_1\|_{X^{s,\frac{1}{2}+\epsilon}_{|\tau|=|\xi|}} \,, 
\end{equation}
so that (\ref{41}) follows from (\ref{40}).
Assuming from now on $||\tau_3|-|\xi_3|| \ll |\xi_3|$ we have to prove
\begin{equation}
\label{42}
 \int_* m(\xi_1,\xi_2,\xi_3,\tau_1,\tau_2,\tau_3) \prod_{i=1}^3 \widehat{u}_i(\xi_i,\tau_i) d\xi d\tau \lesssim \prod_{i=1}^3 \|u_i\|_{L^2_{xt}} 
 \end{equation}
where
$$ m= \frac{(|\tau_2|+|\tau_3|) \chi_{||\tau_3|-|\xi_3|| \ll |\xi_3|}}{\langle \xi_1 \rangle^{1-l} \langle \tau_1 \rangle^{\frac{1}{2}-\epsilon} \langle \xi_2 \rangle^{s} \langle \tau_2 \rangle^{\frac{1}{2}+\epsilon} \langle \xi_3 \rangle^s \langle |\tau_3|-|\xi_3|\rangle^{\frac{1}{2}+\epsilon}} \, . $$
Since $\langle \tau_3 \rangle \sim \langle \xi_3 \rangle$ and $\tau_1+\tau_2+\tau_3=0$ we have 
\begin{equation}
\label{43}
|\tau_2| + |\tau_3| \lesssim \langle \tau_1 \rangle^{\frac{1}{2}-\epsilon} \langle \tau_2 \rangle^{\frac{1}{2}+\epsilon} +\langle \tau_1 \rangle^{\frac{1}{2}-\epsilon} \langle \xi_3 \rangle^{\frac{1}{2}+\epsilon} +\langle \tau_2 \rangle^{\frac{1}{2}+\epsilon} \langle \xi_3 \rangle^{\frac{1}{2}-\epsilon} \, .
\end{equation}
Thus (\ref{42}) is a consequence of the following three estimates:
\begin{align*}
\big|\int\int uvw dx dt\big| &\lesssim \|u\|_{X^{1-l,0}_{\tau=0}} \|v\|_{X^{s,0}_{\tau=0}} \|w\|_{X^{s,\frac{1}{2}+\epsilon}_{|\tau|=|\xi|}}
\\
\big|\int\int uvw dx dt\big| &\lesssim \|u\|_{X^{1-l,0}_{\tau=0}} \|v\|_{X^{s,\frac{1}{2}+\epsilon}_{\tau=0}} \|w\|_{X^{s-\frac{1}{2}-\epsilon,\frac{1}{2}+\epsilon}_{|\tau|=|\xi|}} \\
\big|\int\int uvw dx dt\big| &\lesssim \|u\|_{X^{1-l,\frac{1}{2}-\epsilon}_{\tau=0}} \|v\|_{X^{s,0}_{\tau=0}} \|w\|_{X^{s-\frac{1}{2}+\epsilon,\frac{1}{2}+\epsilon}_{|\tau|=|\xi|}} \, , 
\end{align*}
which easily follow from Sobolev's multiplication law (\ref{SML}) using $s > \frac{1}{4} + \frac{l}{2}$ and $ s > l-\frac{1}{2} $ .

We now come to the proof of (\ref{27}) and remark that we may assume now that both functions $\phi_1$ and $\phi_2$ are supported in $||\tau|-|\xi|| \ll |\xi|$ , because otherwise (\ref{27}) is an immediate consequence of (\ref{41}) and (\ref{41'}). Thus (\ref{27}) follows if we can prove the following estimate:
$$ \int_* m(\xi_1,\xi_2,\xi_3,\tau_1,\tau_2,\tau_3) \prod_{i=1}^3 \widehat{u}_i(\xi_i,\tau_i) d\xi d\tau \lesssim \prod_{i=1}^3 \|u_i\|_{L^2_{xt}} \, , $$
where
$$m= \frac{|\tau_3|\chi_{||\tau_2|-|\xi_2|| \ll |\xi_2|} \chi_{||\tau_3|-|\xi_3|| \ll |\xi_3|}}{\langle \xi_1 \rangle^{1-l} \langle \tau_1 \rangle^{\frac{1}{2}-\epsilon} \langle \xi_2 \rangle^s \langle |\tau_2|-|\xi_2| \rangle^{\frac{1}{2}+\epsilon} \langle \xi_3 \rangle^s \langle |\tau_3|-|\xi_3|\rangle^{\frac{1}{2}+\epsilon}} \, . $$
Since $\langle \tau_3 \rangle \sim \langle \xi_3 \rangle$ , $\langle \tau_2 \rangle \sim \langle \xi_2 \rangle$ and $\tau_1+\tau_2+\tau_3=0$ we obtain
$$
|\tau_3| \lesssim \langle \tau_1 \rangle^{\frac{1}{2}-\epsilon} \langle \xi_3 \rangle^{\frac{1}{2}+\epsilon} +\langle \xi_2 \rangle^{\frac{1}{2}-\epsilon} \langle \xi_3 \rangle^{\frac{1}{2}+\epsilon} \, . $$
The first term is taken care of by the estimate
$$\big|\int \int uvw dx dt\big| \lesssim \|u\|_{X^{1-l,0}_{\tau=0}} \|v\|_{X^{s,\frac{1}{2}+\epsilon}_{|\tau|=|\xi|}} \|w\|_{X^{s-\frac{1}{2}-\epsilon,\frac{1}{2}+\epsilon}_{|\tau|=|\xi|}} \, ,$$
which follows from Prop. \ref{Prop.2} under the assumptions $s> \frac{1}{8} + \frac{l}{2}$ , $s>\frac{1}{4}+ \frac{l}{4} $ , $ s > \frac{1}{2}$ and $ s > l-\frac{1}{2} $ .\\
In order to treat the second term on the right hand side we have to show
\begin{equation}
\label{43'}
\int_*  \frac{\widehat{u}_1(\xi_1,\tau_1)\langle \xi_1 \rangle^{l-1}\widehat{u}_2(\xi_2,\tau_2)\widehat{u}_3(\xi_3,\tau_3)}{ \langle \tau_1 \rangle^{\frac{1}{2}-\epsilon} \langle \xi_2 \rangle^{s-\frac{1}{2}} \langle |\tau_2|-|\xi_2| \rangle^{\frac{1}{2}+\epsilon} \langle \xi_3 \rangle^{s-\frac{1}{2}} \langle |\tau_3|-|\xi_3|\rangle^{\frac{1}{2}+\epsilon}}  d\xi d\tau \lesssim \prod_{i=1}^3 \|u_i\|_{L^2_{xt}} \, , 
\end{equation}
which is equivalent to
$$ \big|\int \int uvw dx dt\big| \lesssim \|u\|_{X^{1-l,\frac{1}{2}-\epsilon}_{\tau=0}} \|v\|_{X^{s-\frac{1}{2}+\epsilon,\frac{1}{2}+\epsilon}_{|\tau|=|\xi|}} \|w\|_{X^{s-\frac{1}{2}-\epsilon,\frac{1}{2}+\epsilon}_{|\tau|=|\xi|}} \, .
$$
We consider first the case $l \le 1$.
By H\"older's inequality we obtain
$$ \big| \int \int uvw dx dt \big| \le \|u\|_{L^q_x L^r_t} \|v\|_{L^p_x L^z_t} \|w\|_{L^p_x L^z_t} \, $$
where we choose $\frac{1}{q} = \frac{l}{2}$ , $\frac{1}{p} = \frac{1}{2} - \frac{l}{4}$ , $\frac{1}{r} = \epsilon$ , $\frac{1}{z} = \frac{1}{2}-\frac{\epsilon}{2}$ , so that we obtain by Sobolev
$$ \|u\|_{L^q_x L^r_t} \lesssim \|u\|_{H^{1-l}_x H^{\frac{1}{2}-\epsilon}_t} \lesssim \|u\|_{X^{1-l,\frac{1}{2}-\epsilon}_{\tau = 0}} \, . $$
Because $z=2+$ and $2 \le p \le 6$ (for $l \le \frac{4}{3}$) we may apply Lemma \ref{Lemma} and obtain the desired estimate for $v$ and also $w$ :
$$\|v\|_{L^p_x L^z_t} \lesssim \|v\|_{X^{\frac{1}{2}(\frac{1}{2}-\frac{1}{p})+,\frac{1}{2}+}_{|\tau|=|\xi|}} \le \|v\|_{X^{s-\frac{1}{2}-,\frac{1}{2}+}_{|\tau|=|\xi|}} \, , $$
provided $\frac{1}{2}(\frac{1}{2}-\frac{1}{p}) < s - \frac{1}{2} \, \Longleftrightarrow \, s > \frac{1}{2} + \frac{l}{8} \, .$ Here the decisive lower bound for $s$ is required, namely $l=s > \frac{1}{2} + \frac{1}{14}$ , because we shall see below that we need $l \ge s$ for the estimate (\ref{29}). The proof of (\ref{43'}) in the case $l \le 1$ is complete.\\
Next we consider the case $ l>1 $. The left hand side of (\ref{43'}) is bounded by
\begin{align}
\nonumber
&\int_*  \frac{\widehat{u}_1(\xi_1,\tau_1)\widehat{u}_2(\xi_2,\tau_2)\widehat{u}_3(\xi_3,\tau_3)}{ \langle \tau_1 \rangle^{\frac{1}{2}-\epsilon} \langle \xi_2 \rangle^{s-l+\frac{1}{2}} \langle |\tau_2|-|\xi_2| \rangle^{\frac{1}{2}+\epsilon} \langle \xi_3 \rangle^{s-\frac{1}{2}} \langle |\tau_3|-|\xi_3|\rangle^{\frac{1}{2}+\epsilon}}  d\xi d\tau  \\ \nonumber
& \lesssim 
\nonumber
\int_*  \frac{\widehat{u}_1(\xi_1,\tau_1)\widehat{u}_2(\xi_2,\tau_2)\widehat{u}_3(\xi_3,\tau_3)}{ \langle \tau_1 \rangle^{\frac{1}{2}-\epsilon}  \langle |\tau_2|-|\xi_2| \rangle^{\frac{1}{2}+\epsilon} \langle \xi_3 \rangle^{2s-l} \langle |\tau_3|-|\xi_3|\rangle^{\frac{1}{2}+\epsilon}}  d\xi d\tau \\ \label{43''}
& \lesssim 
\int_*  \frac{\widehat{u}_1(\xi_1,\tau_1)\widehat{u}_2(\xi_2,\tau_2)\widehat{u}_3(\xi_3,\tau_3)}{ \langle \tau_1 \rangle^{\frac{1}{2}-\epsilon}  \langle |\tau_2|-|\xi_2| \rangle^{\frac{1}{2}+\epsilon} \langle \xi_3 \rangle^{\frac{1}{2}+} \langle |\tau_3|-|\xi_3|\rangle^{\frac{1}{2}+\epsilon}}  d\xi d\tau  \, ,
\end{align}
where we assumed w.l.o.g. $|\xi_2| \ge |\xi_3|$ , so that $\langle \xi_1 \rangle \lesssim \langle \xi_2 \rangle$ , and our assumptions $s-l+\frac{1}{2} \ge 0$ and $2s-l > \frac{1}{2}$ . By Sobolev and Lemma \ref{Lemma} we obtain
$$ \|w\|_{L^{\infty}_x L^2_t} \lesssim \|w\|_{H^{\frac{1}{3}+,6}_x L^2_t} \lesssim \|w\|_{X^{\frac{1}{2}+,\frac{1}{2}+}_{|\tau|=|\xi|}} \, $$
which implies
\begin{align*}
|\int uvw \, dx dt | &\lesssim  \|u\|_{L^2_x L^2_t} \|v\|_{L^2_x L^{\infty}_t} \|w\|_{L^{\infty}_x L^2_t} \\
 &\lesssim  \|u\|_{X^{0,\frac{1}{2}-\epsilon}_{\tau=0}} \|v\|_{X^{0,\frac{1}{2}+\epsilon}_{|\tau|=|\xi|}} \|w\|_{X^{\frac{1}{2}+,\frac{1}{2}+}_{|\tau|=|\xi|}} \, . 
\end{align*} 
Thus (\ref{43''}) is bounded by $\prod_{i=1}^3 \|u_i\|_{L^2_{xt}}$ , which completes the proof of 
(\ref{43'}) and  also (\ref{27}). \\
{\bf Proof of (\ref{29}):} This proof is similar to a related estimate for the Yang-Mills equation given by Tao \cite{T1}. We have to show
$$
\int_* m(\xi_1,\xi_2,\xi_3,\tau_1,\tau_2,\tau_3) \prod_{i=1}^3 \widehat{u}_i(\xi_i,\tau_i)  d\xi d\tau \lesssim \prod_{i=1}^3 \|u_i\|_{L^2_{xt}} \, , 
$$
where 
$$ m = \frac{(|\xi_2|+|\xi_3|) \langle \xi_1 \rangle^{s-1} }{\langle |\tau_1|-|\xi_1|) \rangle^{\frac{1}{2}-2\epsilon}\langle \xi_2 \rangle^s \langle |\tau_2| - |\xi_2|\rangle^{\frac{1}{2}+\epsilon} |\xi_3|^{\tilde{\epsilon}} \langle \xi_3 \rangle^{l-\tilde{\epsilon}}\langle \tau_3 \rangle^{\frac{1}{2}+\epsilon-}} \, .$$
Case 1: $|\xi_2| \lesssim |\xi_1|$ ($\Rightarrow$ $|\xi_2|+|\xi_3| \lesssim |\xi_1|$). \\
We ignore the factor $\langle |\tau_1| - |\xi_1| \rangle^{\frac{1}{2}-2\epsilon}$ and use the averaging principle (\cite{T}, Prop. 5.1) to replace $m$ by
$$ m' = \frac{ \langle \xi_1 \rangle^s \chi_{||\tau_2|-|\xi_2||\sim 1} \chi_{|\tau_3| \sim 1}}{ \langle \xi_2 \rangle^s |\xi_3|^{\tilde{\epsilon}} \langle \xi_3 \rangle^{l-\tilde{\epsilon}}} \, . $$
Let now $\tau_2$ be restricted to the region $\tau_2 =T + O(1)$ for some integer $T$. Then $\tau_1$ is restricted to $\tau_1 = -T + O(1)$, because $\tau_1 + \tau_2 + \tau_3 =0$, and $\xi_2$ is restricted to $|\xi_2| = |T| + O(1)$. The $\tau_1$-regions are essentially disjoint for $T \in {\mathbb Z}$ and similarly the $\tau_2$-regions. Thus by Schur's test (\cite{T}, Lemma 3.11) we only have to show
\begin{align*}
 &\sup_{T \in {\mathbb Z}} \int_* \frac{\langle \xi_1 \rangle^s \chi_{\tau_1=-T+O(1)} \chi_{\tau_2=T+O(1)} \chi_{|\tau_3|\sim 1} \chi_{|\xi_2|=|T|+O(1)}}{\langle \xi_2 \rangle^s |\xi_3|^{\tilde{\epsilon}} \langle \xi_3 \rangle^{l-\tilde{\epsilon}}} \prod_{i=1} \widehat{u}_i(\xi_i,\tau_i)  d\xi d\tau  \\
 & \hspace{25em} \lesssim \prod_{i=1}^3 \|u_i\|_{L^2_{xt}} \, . 
\end{align*}
The $\tau$-behaviour of the integral is now trivial, thus we reduce to
\begin{equation}
\label{55}
\sup_{T \in {\mathbb N}} \int_{\sum_{i=1}^3 \xi_i =0}  \frac{ \langle \xi_1 \rangle^s \chi_{|\xi_2|=T+O(1)}}{ \langle T \rangle^s |\xi_3|^{\tilde{\epsilon}} \langle \xi_3 \rangle^{l-\tilde{\epsilon}}} \widehat{f}_1(\xi_1)\widehat{f}_2(\xi_2)\widehat{f}_3(\xi_3)d\xi \lesssim \prod_{i=1}^3 \|f_i\|_{L^2_x} \, .
\end{equation}
Assuming now $|\xi_3| \le |\xi_1|$ (the other case being simpler) 
it only remains to consider the following two cases: \\
Case 1.1: $|\xi_1| \sim |\xi_3| \gtrsim T$. We now use our assumption $l \ge s$ , so that it suffices to show 
$$
\sup_{T \in {\mathbb N}} \int_{\sum_{i=1}^3 \xi_i =0}  
\frac{\chi_{|\xi_2|=T+O(1)}}{ T^l} \widehat{f}_1(\xi_1)\widehat{f}_2(\xi_2)\widehat{f}_3(\xi_3)d\xi \lesssim \prod_{i=1}^3 \|f_i\|_{L^2_x} \, . $$
The l.h.s. is bounded by
\begin{align*} 
& \sup_{T \in{\mathbb N}} \frac{1}{T^l} \|f_1\|_{L^2} \|f_3\|_{L^2} \| {\mathcal F}^{-1}(\chi_{|\xi|=T+O(1)} \widehat{f}_2)\|_{L^{\infty}({\mathbb R}^2)} \\
&\lesssim \sup_{T \in{\mathbb N}} \frac{1}{ T^l} 
\|f_1\|_{L^2} \|f_3\|_{L^2} \| \chi_{|\xi|=T+O(1)} \widehat{f}_2\|_{L^1({\mathbb R}^2)} \\
&\lesssim \hspace{-0.1em}\sup_{T \in {\mathbb N}} \frac{T^{\frac{1}{2}}}{T^l}  \prod_{i=1}^3 \|f_i\|_{L^2} \lesssim\hspace{-0.1em}
\prod_{i=1}^3 \|f_i\|_{L^2} \, ,
\end{align*}
because one easily calculates that $l > \frac{1}{2}$ under our assumptions. \\
Case 1.2: $|\xi_1| \sim T \gtrsim |\xi_3|$. 
In this case it suffices to show
$$
\sup_{T \in {\mathbb N}} \int_{\sum_{i=1}^3 \xi_i =0}  \frac{\chi_{|\xi_2|=T+O(1)}}{|\xi_3|^{\tilde{\epsilon}} \langle \xi_3 \rangle^{l-\tilde{\epsilon}}} \widehat{f}_1(\xi_1)\widehat{f}_2(\xi_2)\widehat{f}_3(\xi_3)d\xi \lesssim \prod_{i=1}^3 \|f_i\|_{L^2_x} \, .
$$
Case 1.2.1: $|\xi_3| \ge 1$ . An elementary calculation shows  that the l.h.s. is bounded by
\begin{align*}
 \sup_{T \in{\mathbb N}} \| \chi_{|\xi|=T+O(1)} \ast \langle \xi \rangle^{-2l}\|^{\frac{1}{2}}_{L^{\infty}(\mathbb{R}^2)} \prod_{i=1}^3 \|f_i\|_{L^2_x} \lesssim \prod_{i=1}^3 \|f_i\|_{L^2_x} \, ,
\end{align*}
using as in case 1.1 that $l > \frac{1}{2}$ .
\\
Case 1.2.2: $|\xi_3| \le 1$ . 
The l.h.s. is crudely estimated by
\begin{align*}
 &\sup_{T \in{\mathbb N}} \| \chi_{|\xi|=T+O(1)} \ast \chi_{|\xi|\le 1} |\xi|^{-2\tilde{\epsilon}}\|^{\frac{1}{2}}_{L^{\infty}(\mathbb{R}^2)} \prod_{i=1}^3 \|f_i\|_{L^2_x} \\
 &\hspace{5em}\lesssim (\int_{|\xi|\le 1} |\xi|^{-2\tilde{\epsilon}} d\xi)^{\frac{1}{2}} \prod_{i=1}^3 \|f_i\|_{L^2_x} \lesssim \prod_{i=1}^3 \|f_i\|_{L^2_x} \, .
\end{align*}
Case 2. $|\xi_1| \ll |\xi_2|$ ($\Rightarrow$ $|\xi_2|+|\xi_3| \lesssim |\xi_2|$). \\
Exactly as in case 1 we reduce to
$$
\sup_{T \in {\mathbb N}} \int_{\sum_{i=1}^3 \xi_i =0}  \frac{ \langle T \rangle^{1-s} \chi_{|\xi_2|=T+O(1)}}{ \langle \xi_1 \rangle^{1-s} |\xi_3|^{\tilde{\epsilon}} \langle \xi_3 \rangle^{l-\tilde{\epsilon}}} \widehat{f}_1(\xi_1)\widehat{f}_2(\xi_2)\widehat{f}_2(\xi_3)d\xi \lesssim \prod_{i=1}^3 \|f_i\|_{L^2_x} \, .
$$
Using $|\xi_3| \sim |\xi_2| \sim T \gg |\xi_1|$ and  $l > \frac{1}{2}$ we crudely estimate:
\begin{align*}
\frac{T^{1-s}}{\langle \xi_1 \rangle^{1-s} |\xi_3|^{\tilde{\epsilon}} \langle \xi_3 \rangle^{l-\tilde{\epsilon}} } \sim \frac{T^{1-s}}{\langle \xi_1 \rangle^{1-s} T^l} \lesssim \frac{T^{1-s}}{\langle \xi_1 \rangle^{1-s} \langle \xi_1 \rangle^{s-\frac{1}{2}+} T^{l-s+\frac{1}{2}-}} \lesssim \frac{1}{\langle \xi_1 \rangle^{\frac{1}{2}+}} \, . 
\end{align*}
Thus we reduce to
$$
\sup_{T \in {\mathbb N}} \int_{\sum_{i=1}^3 \xi_i =0}  
\frac{\chi_{|\xi_2|=T+O(1)}}{ \langle \xi_1 \rangle^{\frac{1}{2}+}} \widehat{f}_1(\xi_1)\widehat{f}_2(\xi_2)\widehat{f}_3(\xi_3)d\xi \lesssim \prod_{i=1}^3 \|f_i\|_{L^2_x} \, , $$
which can be shown as in Case 1.2. The proof of (\ref{29}) is complete. \\
{\bf Proof of (\ref{30}):} Assume first that $r \le 1$. We estimate by Sobolev's multiplication law (\ref{SML}) using $ s > \frac{1}{2}$ :
\begin{align*}
&\| A \phi_1 \phi_2 \|_{X^{r-1,-\frac{1}{2}+2\epsilon}_{|\tau|=|\xi|}} \lesssim \|A \phi_1 \phi_2\|_{L^2_t H^{r-1}_x} \lesssim \|A\|_{L^6_t H^r_x} \|\phi_1 \phi_2\|_{L^3_t H^{0+}_x} \\
&\lesssim \|A\|_{L^6_t H^r_x} \|\phi_1\|_{L^6_t H^{\frac{1}{2}+}_x} \|\phi_2\|_{L^6_t H^{\frac{1}{2}+}_x} \lesssim \|A\|_{X^{r,\frac{3}{4}+\epsilon}_{|\tau|=|\xi|}} \|\phi_1\|_{X^{s,\frac{1}{2}+\epsilon}_{|\tau|=|\xi|}}
\|\phi_2\|_{X^{s,\frac{1}{2}+\epsilon}_{|\tau|=|\xi|}} \, .
\end{align*}
Especially we have
\begin{align*}
\| A \phi_1 \phi_2 \|_{X^{0,-\frac{1}{2}+2\epsilon}_{|\tau|=|\xi|}}   \lesssim \|A\|_{X^{1,\frac{3}{4}+\epsilon}_{|\tau|=|\xi|}} \|\phi_1\|_{X^{\frac{1}{2}+,\frac{1}{2}+\epsilon}_{|\tau|=|\xi|}}
\|\phi_2\|_{X^{\frac{1}{2}+,\frac{1}{2}+\epsilon}_{|\tau|=|\xi|}} \, .
\end{align*}
The fractional Leibniz rule implies for $r>1$ :
\begin{align*}
\| A \phi_1 \phi_2 \|_{X^{r-1,-\frac{1}{2}+2\epsilon}_{|\tau|=|\xi|}} &  \lesssim \|A\|_{X^{r,\frac{3}{4}+\epsilon}_{|\tau|=|\xi|}} \|\phi_1\|_{X^{r-\frac{1}{2}+,\frac{1}{2}+\epsilon}_{|\tau|=|\xi|}}
\|\phi_2\|_{X^{r-\frac{1}{2}+,\frac{1}{2}+\epsilon}_{|\tau|=|\xi|}} \\
&   \lesssim \|A\|_{X^{r,\frac{3}{4}+\epsilon}_{|\tau|=|\xi|}} \|\phi_1\|_{X^{s,\frac{1}{2}+\epsilon}_{|\tau|=|\xi|}}
\|\phi_2\|_{X^{s,\frac{1}{2}+\epsilon}_{|\tau|=|\xi|}}
\end{align*}
by our assumption $r-\frac{1}{2} < s$.\\
Assume now again that $r \le 1$ .
We obtain
\begin{align*}
&\|A \phi_1 \phi_2\|_{L^2_t H^{r-1}_x} 
\lesssim \|A \phi_1 \phi_2\|_{L^2_t L^p_x} \lesssim \|A\|_{L^6_t L^4}  \|\phi_1\|_{L^6_t L^q_x} \|\phi_2\|_{L^6_t L^q_x}\\
& \lesssim \|A\|_{L^6_t \dot{H}^{\frac{1}{2}}} \|\phi_1\|_{L^6_t H^{\frac{1}{4}+\frac{r}{2}}_x} \|\phi_1\|_{L^6_t H^{\frac{1}{4}+\frac{r}{2}}_x} 
 \lesssim \|\nabla|^{\tilde{\epsilon}} A\|_{L^6_t H^{l-\tilde{\epsilon}}_x} \|\phi_1\|_{L^6_t H^s_x}\|\phi_1\|_{L^6_t H^s_x} \\
 & \lesssim \| |\nabla|^{\tilde{\epsilon}} A\|_{X^{l-\tilde{\epsilon},\frac{1}{2}+\epsilon-}_{\tau=0}} \|\phi_1\|_{X^{s,\frac{1}{2}+\epsilon}_{|\tau|=|\xi|}} \|\phi_1\|_{X^{s,\frac{1}{2}+\epsilon}_{|\tau|=|\xi|}}
 \, ,
\end{align*}
where $\frac{1}{p}= 1 - \frac{r}{2}$ , $\frac{1}{q} = \frac{3}{8}-\frac{r}{4}$ , so that by Sobolev we obtain $L^p_x \hookrightarrow H^{r-1}_x$ , $H^{\frac{1}{4}+\frac{r}{2}}_x \hookrightarrow L^q_x$ and $\dot{H}^{\frac{1}{2}}_x \hookrightarrow L^4_x $ . We used $l \ge \frac{1}{2}$ and $s \ge \frac{r}{2}+\frac{1}{4}$ . 
Especially we obtain for $r=1$ :
\begin{align*}
\|A \phi_1 \phi_2\|_{L^2_t L^2_x} 
\lesssim \| |\nabla|^{\tilde{\epsilon}} A\|_{X^{l-\tilde{\epsilon},\frac{1}{2}+\epsilon-}_{\tau=0}} \|\phi_1\|_{L^6_t H^{\frac{3}{4}}_x} \|\phi_1\|_{L^6_t H^{\frac{3}{4}}_x}
 \, .
\end{align*}
Next we consider the case $r=\frac{3}{2}$ . By the fractional Leibniz rule we obtain
\begin{align*}
&\|A \phi_1 \phi_2\|_{L^2_t H^{\frac{1}{2}}_x} \\
&\lesssim \|A\|_{L^6_t H^{\frac{1}{2},2+}}  \|\phi_1\|_{L^6_t L^{\infty -}_x} \|\phi_2\|_{L^6_t L^{\infty -}_x}
 +   \|A\|_{L^6_t L^{4+}_x}  \|\phi_1\|_{L^6_t H^{\frac{1}{2},4-}_x} \|\phi_2\|_{L^6_t L^{\infty -}_x}  \\
& \hspace{1em} +  \|A\|_{L^6_t L^{4+}_x}   \|\phi_1\|_{L^6_t L^{\infty -}_x} \|\phi_2\|_{L^6_t H^{\frac{1}{2},4-}_x}\\
& \lesssim \| |\nabla|^{\tilde{\epsilon}} A\|_{X^{\frac{1}{2},\frac{1}{2}+\epsilon-}_{\tau=0}} \|\phi_1\|_{L^6_t H^{1-}_x} \|\phi_2\|_{L^6_t H^{1-}_x} \\
& \lesssim \| |\nabla|^{\tilde{\epsilon}} A\|_{X^{\frac{1}{2},\frac{1}{2}+\epsilon-}_{\tau=0}}
\|\phi_1\|_{X^{1-,\frac{1}{2}+\epsilon}_{|\tau|=|\xi|}} \|\phi_2\|_{X^{1-,\frac{1}{2}+\epsilon}_{|\tau|=|\xi|}}
 \, ,
\end{align*}
This is enough, because $s \ge 1$ and $ l > \frac{1}{2} $ . \\
By bilinear interpolation between the cases $r=1$ and $r=\frac{3}{2}$ we easily obtain for $1<r<\frac{3}{2}$ :
\begin{align*}
&\|A \phi_1 \phi_2\|_{L^2_t H^{r-1}_x} \\
& \lesssim \| |\nabla|^{\tilde{\epsilon}} A\|_{X^{l-\tilde{\epsilon},\frac{1}{2}+\epsilon-}_{\tau=0}}
\|\phi_1\|_{L^6_t H^{\frac{1}{4}+\frac{r}{2}}} \|\phi_2\|_{L^6_t H^{\frac{1}{4}+\frac{r}{2}}}\\
& \lesssim \| |\nabla|^{\tilde{\epsilon}} A\|_{X^{l-\tilde{\epsilon},\frac{1}{2}+\epsilon-}_{\tau=0}} \|\phi_1\|_{X^{s,\frac{1}{2}+\epsilon}_{|\tau|=|\xi|}} \|\phi_1\|_{X^{s,\frac{1}{2}+\epsilon}_{|\tau|=|\xi|}}
 \, .
\end{align*}
under our assumption $s > \frac{r}{2} + \frac{1}{4}$ . \\
The remaining case $r > \frac{3}{2}$ follows from the case $r = \frac{3}{2}$ by the fractional Leibniz rule:
\begin{align*}
\|A \phi_1 \phi_2\|_{L^2_t H^{r-1}_x} 
& \lesssim \| |\nabla|^{\tilde{\epsilon}} A\|_{X^{r-1,\frac{1}{2}+\epsilon-}_{\tau=0}}
\|\phi_1\|_{X^{r-\frac{1}{2},\frac{1}{2}+\epsilon}_{|\tau|=|\xi|}} \|\phi_2\|_{X^{r-\frac{1}{2},\frac{1}{2}+\epsilon}_{|\tau|=|\xi|}}
 \\
& \lesssim \| |\nabla|^{\tilde{\epsilon}} A\|_{X^{l-\tilde{\epsilon},\frac{1}{2}+\epsilon-}_{\tau=0}} \|\phi_1\|_{X^{s,\frac{1}{2}+\epsilon}_{|\tau|=|\xi|}} \|\phi_1\|_{X^{s,\frac{1}{2}+\epsilon}_{|\tau|=|\xi|}} \, ,
\end{align*}
where we used  $r-\frac{1}{2} \le s \le l$ .

{\bf Proof of (\ref{30'}):} Assume first $ s \le 1$. By Sobolev's multiplication rule (\ref{SML}) and $ l \ge s > \frac{1}{2}$ we obtain
\begin{align*}
\|A_1 A_2 \phi\|_{L^2_t H^{s-1}_x} &\lesssim \|A_1 A_2\|_{L^2_t H^{0+}_x} \|\phi\|_{L^{\infty}_t H^s_x}             \lesssim \|A_1\|_{L^4_t H^{0+,4}_x } \|A_2\|_{L^4_t H^{0+,4}_x} \|\phi\|_{L^{\infty}_t H^s_x} \\
& \lesssim \||\nabla|^{\tilde{\epsilon}}A_1\|_{L^4_t H^{\frac{1}{2}-\tilde{\epsilon}+}_x} \||\nabla|^{\tilde{\epsilon}}A_1\|_{L^4_t H^{\frac{1}{2}-\tilde{\epsilon}+}_x}\|\phi\|_{L^{\infty}_t H^s_x} \\
& \lesssim \| |\nabla|^{\tilde{\epsilon}} A_1\|_{X^{l-\tilde{\epsilon},\frac{1}{2}+\epsilon-}_{\tau=0}} \| |\nabla|^{\tilde{\epsilon}} A_2\|_{X^{l-\tilde{\epsilon},\frac{1}{2}+\epsilon-}_{\tau=0}} \|\phi\|_{X^{s,\frac{1}{2}+\epsilon}_{|\tau|=|\xi|}} \, .
\end{align*}
For $s>1$ the Sobolev multiplication law implies :
\begin{align*}
\|A_1 A_2 \phi\|_{L^2_t H^{s-1}_x}  \lesssim \|A_1 A_2\|_{L^3_t H^{s-1}_x} \|\phi\|_{L^6_t H^s_x} \, .
\end{align*}
Now using $l \ge s > 1$ we obtain
\begin{align*}
&\|A_1 A_2\|_{L^3_t H^{s-1}_x} \lesssim  \| \langle \nabla \rangle^{s-1} A_1 A_2 \|_{L^3_t L^2_x} + \| A_1 \langle \nabla \rangle^{s-1} A_2 \|_{L^3_t L^2_x} \\
&\lesssim \|\langle \nabla \rangle^{s-1} A_1\|_{L^6_t L^4_x} \|A_2\|_{L^6_t L^4_x} +\|A_1\|_{L^6_t L^4_x} \|\langle \nabla \rangle^{s-1} A_2\|_{L^6_t L^4_x} \\
&\lesssim  \| |\nabla|^{\frac{1}{2}} A_1\|_{L^6_t H^{s-1}_x} \| |\nabla|^{\frac{1}{2}} A_2 \|_{L^6_t L^2_x} +  \| |\nabla|^{\frac{1}{2}} A_1 \|_{L^6_t L^2_x} \| |\nabla|^{\frac{1}{2}} A_2\|_{L^6_t H^{s-1}_x}\\
&\lesssim \| |\nabla|^{\tilde{\epsilon}} A_1\|_{X^{l-\tilde{\epsilon},\frac{1}{2}+\epsilon-}_{\tau =0}} \| |\nabla|^{\tilde{\epsilon}} A_2\|_{X^{l-\tilde{\epsilon},\frac{1}{2}+\epsilon-}_{\tau =0}} \, ,
\end{align*}
which gives the same bound for $\|A_1 A_2 \phi\|_{L^2_t H^{s-1}_x}$ as in the case $s \le 1$ . 

Next, let us assume first that $s \le \frac{3}{4}$ . 
By Prop. \ref{Prop.2} we obtain
$$ \|A_1 A_2 \phi \|_{X^{s-1,-\frac{1}{2}+2\epsilon}_{|\tau|=|\xi|}} \lesssim \|A_1 A_2\|_{X^{-\frac{1}{4}+,0}_{|\tau|=|\xi|}} \|\phi\|_{X^{s,\frac{1}{2}+\epsilon}_{|\tau|=|\xi|}} \, . $$
It remains to estimate $\|A_1 A_2\|_{X^{-\frac{1}{4}+,0}_{|\tau|=|\xi|}}=\|A_1 A_2\|_{L^2_t H^{-\frac{1}{4}+}_x}$ . On the one hand we obtain for $l>\frac{1}{2}$ and $r>\frac{1}{4}$ by Sobolev
\begin{align*}
\|A_1 A_2\|_{L^2_t H^{-\frac{1}{4}+}_x} & \lesssim \|A_1 A_2\|_{L^2_t L^{\frac{8}{5}+}_x} \lesssim \|A_1\|_{L^4_t L^4_x} \| A_2\|_{L^4_t L^{\frac{8}{3}+}_x} \lesssim \|A_1\|_{L^4_t \dot{H}^{\frac{1}{2}}_x} \| A_2\|_{L^4_t H^{\frac{1}{4}+}_x} \\
&\lesssim \| |\nabla|^{\tilde{\epsilon}} A_1\|_{L^4_t H^{l-\tilde{\epsilon}}_x} \|A_2\|_{L^4_t H^r_x}  
\lesssim \| |\nabla|^{\tilde{\epsilon}} A_1\|_{X^{l-\tilde{\epsilon},\frac{1}{2}+\epsilon-}_{\tau =0}} \|A_2\|_{X^{r,\frac{3}{4}+\epsilon}_{|\tau|=|\xi|}} \, .
\end{align*}
On the other hand we use Prop. \ref{Prop.2} again and obtain for $r > \frac{1}{4}$ :
$$\|A_1 A_2\|_{X^{-\frac{1}{4}+,0}_{|\tau|=|\xi|}} \lesssim \|A_1\|_{X^{\frac{1}{4}+,\frac{1}{2}+}_{|\tau|=|\xi|}}
\|A_2\|_{X^{\frac{1}{4}+,\frac{1}{2}+}_{|\tau|=|\xi|}} \lesssim \|A_1\|_{X^{r,\frac{3}{4}+\epsilon}_{|\tau|=|\xi|}} \|A_2\|_{X^{r,\frac{3}{4}+\epsilon}_{|\tau|=|\xi|}} \, , $$
which completes the proof in the case $ s \le \frac{3}{4} $ .

Next, we assume $ s > \frac{3}{4}$ . By Prop. \ref{Prop.2} we obtain
$$ \|A_1 A_2 \phi \|_{X^{s-1,-\frac{1}{2}+2\epsilon}_{|\tau|=|\xi|}} \lesssim \|A_1 A_2\|_{X^{s-1,0}_{|\tau|=|\xi|}} \|\phi\|_{X^{s,\frac{1}{2}+\epsilon}_{|\tau|=|\xi|}} \, . $$
It remains to estimate $\|A_1 A_2\|_{X^{s-1,0}_{|\tau|=|\xi|}}=\|A_1 A_2\|_{L^2_t H^{s-1}_x}$ . On the one hand we apply Prop. \ref{Prop.2} again , use $r \ge s -\frac{1}{2}$ and obtain
$$ \|A_1 A_2\|_{X^{s-1,0}_{|\tau|=|\xi|}} \lesssim \|A_1\|_{X^{r,\frac{3}{4}+\epsilon}_{|\tau|=|\xi|}}  \|A_2\|_{X^{r,\frac{3}{4}+\epsilon}_{|\tau|=|\xi|}} \, . $$
On the other hand, if $\frac{3}{4} < s \le 1$ we crudely estimate using $l\ge s > \frac{3}{4}$ and $r \ge s-\frac{1}{2} > \frac{1}{4}$ :
\begin{align*}
\|A_1 A_2\|_{L^2_t H^{s-1}_x} &\le \|A_1 A_2\|_{L^2_t L^2_x} \le \|A_1\|_{L^4_t L^8_x} \|A_2\|_{L^4_t L^{\frac{8}{3}}_x} \lesssim \| |\nabla|^{\frac{3}{4}} A_1\|_{L^4_t L^2_x} \|A_2\|_{L^4_t H^{\frac{1}{4}}_x} \\
&\lesssim \| |\nabla|^{\tilde{\epsilon}} A_1\|_{X^{l-\tilde{\epsilon},\frac{1}{2}+\epsilon-}_{\tau =0}} \|A_2\|_{X^{r,\frac{3}{4}+\epsilon}_{|\tau|=|\xi|}} \, .
\end{align*}
If $ s > 1 $ we use $l \ge s > 1$ and $r \ge s-\frac{1}{2}  > \frac{1}{2}$ and obtain
\begin{align*}
&\|A_1 A_2\|_{L^2_t H^{s-1}_x} \lesssim  \| \langle \nabla \rangle^{s-1} A_1 A_2 \|_{L^2_t L^2_x} + \| A_1 \langle \nabla \rangle^{s-1} A_2 \|_{L^2_t L^2_x} \\
&\lesssim \|\langle \nabla \rangle^{s-1} A_1\|_{L^4_t L^4_x} \|A_2\|_{L^4_t L^4_x} +\|A_1\|_{L^4_t L^4_x} \|\langle \nabla \rangle^{s-1} A_2\|_{L^4_t L^4_x} \\
&\lesssim  \| |\nabla|^{\frac{1}{2}} A_1\|_{L^4_t H^{s-1}_x} \| |\nabla|^{\frac{1}{2}} A_2 \|_{L^4_t L^2_x} +  \| |\nabla|^{\frac{1}{2}} A_1 \|_{L^4_t L^2_x} \| |\nabla|^{\frac{1}{2}} A_2\|_{L^4_t H^{s-1}_x}\\
&\lesssim \| |\nabla|^{\tilde{\epsilon}} A_1\|_{X^{l-\tilde{\epsilon},\frac{1}{2}+\epsilon-}_{\tau =0}} \|  A_2\|_{X^{r,\frac{3}{4}+\epsilon}_{|\tau|=|\xi|}}\, ,
\end{align*}
which completes the proof in the case $s>\frac{3}{4}$ and also
the proof of (\ref{30'}) and  part 1 of Theorem \ref{Theorem}. 
\vspace{1em}\\
{\bf Proof of part 2 of Theorem \ref{Theorem} :} The claimed regularity of the solution clearly holds. Let us now assume that the solution fulfills
$$\phi_{\pm} \, , \, A^{df}_{\pm} \in C^0([0,T],H^1({\mathbb R}^2)) \, , \, |\nabla|^{\tilde{\epsilon}} A^{cf} \in C^0([0,T],H^{1-\tilde{\epsilon}}({\mathbb R}^2)) \, . $$
We want to show that such a solution belongs to a space where uniqueness holds by part 1 of the theorem. \\
{\bf Step 1:} $ A^{df}_{\pm} \in X^{1-,1-}_{\pm}[0,T]$ .  \\
We drop $[0,T]$ from all the spaces in the sequel.
Interpolation between Strichartz' estimate (\ref{Str}) and  $\|u\|_{L^2_{xt}} = \|u\|_{X^{0,0}_{\pm}}$ gives $\|u\|_{L^{2+}_t L^{2+}_x} \lesssim \|u\|_{X^{0+,0+}_{|\tau|=|\xi|}} $ , thus by duality $\|u\|_{X^{0-,0-}_{|\tau|=|\xi|}} \lesssim \|u\|_{L^{2-}_t L^{2-}_x} $ . Consequently,
\begin{align*}
\|\phi \overline{\nabla \phi} \|_{X^{0-,0-}_{\pm}} &
\lesssim \|\phi \overline{\nabla \phi} \|_{L^{2-}_t L^{2-}_x}
\lesssim \|\phi\|_{L^{\infty}_t L^{\infty -}_x} \|\nabla \phi\|_{L^{\infty}_t L^2_x} T^{\frac{1}{2}+} \\ 
&\lesssim \|\phi\|_{L^{\infty}_t \dot{H}^{1-}_x} \|\phi\|_{L^{\infty}_t \dot{H}^1_x} T^{\frac{1}{2}+} < \infty 
\end{align*}
Moreover
\begin{align*}
\|A |\phi|^2 \|_{X^{0-,0-}_{\pm}} & \lesssim \|A |\phi|^2\|_{L^{2-}_t L^{2-}_x} \lesssim \|A\|_{L^{\infty}_t L^{6-}_x} \|\phi\|_{L^{\infty}_t L^{6-}_x}^2 T^{\frac{1}{2}+}  \\
& \lesssim \|A\|_{L^{\infty}_t \dot{H}^{\frac{2}{3}-}_x} \|\phi\|_{L^{\infty}_t \dot{H}^{\frac{2}{3}-}_x} T^{\frac{1}{2}+} < \infty \, .
\end{align*}
By (\ref{1.12*}) we obtain the desired regularity. \\
{\bf Step 2:} $\phi_{\pm} \in X^{1-,1-}_{\pm}[0,T]$ . \\
Using (\ref{1.13*}) this leads to the same estimates as in step 1. \\
{\bf Step 3:} $|\nabla|^{\tilde{\epsilon}} A^{cf} \in X^{\frac{3}{4}-\tilde{\epsilon},\frac{1}{2}+}_{\tau =0} $ . \\
Using (\ref{1.11*}) and step 2 it suffices to show
$$
 \| |\nabla|^{-1+\tilde{\epsilon}}(\phi \overline{\partial_t \phi})\|_{X^{\frac{3}{4}-\tilde{\epsilon},-\frac{1}{2}+}_{\tau=0}} \lesssim \|\phi\|_{X^{1-,1-}_{|\tau| = |\xi|}}^2 \, . $$
Replacing as before $|\nabla|^{-1+\tilde{\epsilon}}$ by $\langle \nabla \rangle^{-1+\tilde{\epsilon}}$ this reduces to
\begin{align*}
 \int_* \frac{\widehat{u_1}(\xi_1,\tau_1)}{\langle \xi_1 \rangle^{1-} \langle |\tau_1| - |\xi_1| \rangle^{1-}} \frac{\widehat{u_2}(\xi_2,\tau_2) |\tau_2|}{\langle \xi_2 \rangle^{1-} \langle |\tau_2| - |\xi_2| \rangle^{1-}} 
\frac{\widehat{u_3}(\xi_3,\tau_3)}{\langle \xi_3 \rangle^{\frac{1}{4}} \langle \tau_3 \rangle^{\frac{1}{2}-}}
 \, d\xi d\tau  \lesssim \prod_{i=1}^3 \|u_i\|_{L^2_{xt}} \, .
\end{align*}
Case 1: $|\tau_2| \lesssim |\xi_2|$ . Using 
$$|\tau_2| \lesssim \langle \xi_2 \rangle^{1-} \langle \xi_2 \rangle^{0+} \lesssim \langle \xi_2 \rangle^{1-}(\langle \xi_1 \rangle^{0+} + \langle \xi_3 \rangle^{0+}) $$
we reduce to
$$
\int_* \frac{\widehat{u_1}(\xi_1,\tau_1)}{\langle \xi_1 \rangle^{1--} \langle |\tau_1| - |\xi_1| \rangle^{1-}} \frac{\widehat{u_2}(\xi_2,\tau_2)}{ \langle |\tau_2| - |\xi_2| \rangle^{1-}} 
\frac{\widehat{u_3}(\xi_3,\tau_3)}{\langle \xi_3 \rangle^{\frac{1}{4}-} \langle \tau_3 \rangle^{\frac{1}{2}-}}
 \, d\xi d\tau   \lesssim \prod_{i=1}^3 \|u_i\|_{L^2_{xt}} \, ,
$$
which holds by Sobolev. \\
Case 2: $|\tau_2| \gg |\xi_2|$ and $|\xi_1| \gtrsim |\tau_1|$ . In this case we have
$$ |\tau_2| \lesssim ||\tau_2|-|\xi_2||^{1-} ( \langle \tau_3 \rangle^{0+} + \langle \xi_1 \rangle^{0+}) \, ,
$$
so that it suffices to show
$$
\int_* \frac{\widehat{u_1}(\xi_1,\tau_1)}{\langle \xi_1 \rangle^{1--} \langle |\tau_1| - |\xi_1| \rangle^{1-}} \frac{\widehat{u_2}(\xi_2,\tau_2)}{\langle \xi_2 \rangle^{1-}} 
\frac{\widehat{u_3}(\xi_3,\tau_3)}{\langle \xi_3 \rangle^{\frac{1}{4}} \langle \tau_3 \rangle^{\frac{1}{2}--}}
 \, d\xi d\tau   \lesssim \prod_{i=1}^3 \|u_i\|_{L^2_{xt}} \, ,
$$
which also holds by Sobolev. \\
Case 3: $|\tau_2| \gg |\xi_2|$ and $|\tau_1| \gg |\xi_1|$ . In this case we obtain
$$ |\tau_2| \lesssim \langle|\tau_2|-|\xi_2|\rangle^{1-} ( \langle \tau_3 \rangle^{0+} + \langle |\tau_1|-|\xi_1| \rangle^{0+}) \, ,
$$
which can be handled similarly as case 2. 

The regularity obtained in steps 1-3 is more than sufficient to deduce the uniqueness by an application of part 1 of the theorem. 
\end{proof}

\end{document}